\documentclass{article}
\usepackage{arxiv}
\usepackage{hyperref}
\usepackage[utf8]{inputenc}

\usepackage{setspace}

\usepackage{amssymb,amsmath,amsthm,eucal}

\usepackage[english]{babel}

\usepackage[T1]{fontenc}
\usepackage[utf8]{inputenc}
\numberwithin{equation}{section}

\usepackage{graphicx}
\usepackage{tabularx}
\usepackage{amsfonts,amsmath,amssymb}
\usepackage{booktabs}
\usepackage{float}
\usepackage{subcaption}
\usepackage{multirow}
\usepackage{xcolor}
\usepackage{color}

\newtheorem{tr}{Theorem}
\newtheorem{proposition}{Proposition}
\newtheorem{definition}{Definition}
\newtheorem{lemma}{Lemma}
\newtheorem{remark}{Remark}
\usepackage{hyperref}
\usepackage{tabularx}
\usepackage{booktabs}
\title{Turing-region preservation in matrix-oriented splitting methods for reaction–diffusion systems}

\author{
 Angela Monti \\
  Istituto per le Applicazioni del Calcolo \lq \lq M. Picone\rq \rq \\
  National Research Council (CNR)\\
  via G. Amendola 122/D, Bari, Italy\\
  \texttt{angela.monti@cnr.it} 
  \And
  Fasma Diele \\
  Istituto per le Applicazioni del Calcolo \lq \lq M. Picone\rq \rq \\
  National Research Council (CNR)\\
  via G. Amendola 122/D, Bari, Italy\\
  \texttt{fasma.diele@cnr.it}
  \And
  Carmela Marangi \\
  Istituto per le Applicazioni del Calcolo \lq \lq M. Picone\rq \rq \\
  National Research Council (CNR)\\
  via G. Amendola 122/D, Bari, Italy\\
  \texttt{carmela.marangi@cnr.it}
}

\begin{document}

\maketitle

\begin{abstract}
We develop matrix-oriented formulations of first-order splitting integrators
for two-species reaction--diffusion systems on two-dimensional domains, using
the Gierer--Meinhardt activator--inhibitor system as a canonical benchmark for discrete Turing instability. Exploiting the differential matrix equation
associated with tensor-product spatial discretizations, we obtain the families \(\mathrm{IE}\!-\!\mathcal{S}\) and \(\mathrm{EX}\!-\!\mathcal{S}\), in which the diffusive flow is treated by implicit Euler or exactly, respectively, and the
reaction substep is approximated by explicit, symplectic, adjoint-symplectic, Poisson, and explicit-variant local maps. Starting from the continuous diffusion-driven instability threshold, expressed through the modal relation
\(J_\mu = J^* - \mu D\), we derive the fully discrete modal
amplification matrices and their Jury conditions. These conditions separate the continuous Turing mechanism, carried by the first Jury condition, from purely
discrete effects carried, in particular, by the second. Specializing the analysis to the Gierer--Meinhardt model, we exhibit two opposite pathologies. First, in a continuous Turing-stable regime, IMEX may generate a stable spurious numerical pattern through a violation of the second Jury condition. Second, in a real continuous Turing regime, the adjoint-symplectic family may be spuriously stable and suppress the pattern that IMEX scheme correctly  detects. For IMEX, whose first Jury condition reproduces the sign of the continuous Turing polynomial exactly, we further give an explicit time-step condition that guarantees preservation of the continuous Turing region, and we show that controlling the second Jury condition is needed only in the Turing-stable regime. These examples show that 
 each integrator induces its own discrete Turing region, which should be compared with the continuous one before the scheme is used to interpret numerical patterns; we frame this requirement as the preservation of a qualitative property of the continuous problem, in the spirit of structure-preserving
 numerical integration.
\end{abstract}

\section{Introduction}
\label{intro}

Reaction-diffusion systems are a classical framework for the description of
self-organization phenomena in biology, ecology, chemistry, and related areas.
In particular, diffusion-driven instability provides one of the main mechanisms
through which a spatially homogeneous equilibrium may lose stability and give
rise to non-homogeneous patterns \cite{turing1952chemical}. In most theoretical studies, this mechanism
is analyzed at the continuous level, or after spatial semi-discretization,
through modal conditions involving the linearized diffusion-reaction operator;
see, e.g.,~\cite{murray2003mathematical}. However, the actual numerical
simulation of such systems is produced by a fully discrete dynamical system,
and therefore by a discrete semiflow which may exhibit stability and
instability properties that differ, sometimes substantially, from those of the
continuous model.

We focus on spatially extended systems described by two-species
reaction-diffusion equations in the general form
\begin{equation}
\label{eq:model_intro}
\left\{
\begin{array}{lll}
u_t &=& l_u(u) + f(u,v),\\
v_t &=& l_v(v) + g(u,v),
\end{array}
\right.
\end{equation}
where \(u(x,y;t)\) and \(v(x,y;t)\) represent two interacting species evolving
in
\[
\Omega_T := \Omega \times (0,T),
\]
with \(\Omega \subset \mathbb R^2\) bounded and open. We assume that the
operators \(l_u\) and \(l_v\) are linear diffusion operators, and that
suitable initial and boundary conditions are prescribed.

A central point of this paper is that consistency with the continuous model as
the discretization parameters tend to zero is not, by itself, a sufficient
practical justification for the numerical reproduction of patterns. In
pattern-forming regimes one usually needs long-time integrations in order to
distinguish transient structures from robust asymptotic patterns and to assess
their numerical stability. In such situations, taking extremely small time
steps may become computationally prohibitive, and the discrete dynamics
induced by the time integrator can no longer be regarded as a negligible
perturbation. The numerical method may anticipate the onset of instability,
delay it, suppress it on relevant modes, or generate spurious destabilization
mechanisms that have no counterpart in the continuous problem.

This issue is especially relevant in recent reaction-diffusion-chemotaxis
applications, where long transients, reactivity, and spatial self-organization
play a fundamental role. In particular, recent studies on soil organic carbon
dynamics have shown that microbial activity, chemotaxis, and transient
amplification mechanisms can strongly affect the emergence and interpretation
of spatial patterns~\cite{monti2025socpatterns,diele2025transientreactivity}. One motivation for
the present work is therefore to provide a theoretical justification for the
use of first-order implicit-symplectic and Poisson-type splitting schemes in
such contexts, and to compare them systematically with the IMEX approach,
which remains one of the standard numerical choices for the simulation of
pattern-forming reaction-diffusion systems.

The numerical treatment of \eqref{eq:model_intro} relies on a classical
semi-discretization in space, for instance by finite differences, finite
elements or spectral methods, leading to a system of ordinary differential
equations
\begin{equation}
\label{eq:MOL_intro}
\left\{
\begin{array}{rcl}
\mathbf{\dot u} &=& A_u \mathbf u + \mathbf f(\mathbf u,\mathbf v),\\[2mm]
\mathbf{\dot v} &=& A_v \mathbf v + \mathbf g(\mathbf u,\mathbf v),
\end{array}
\right.
\quad \mathbf u(0)=\mathbf u_0,\quad \mathbf v(0)=\mathbf v_0.
\end{equation}
On rectangular domains, tensor-product discretizations yield matrices of
Kronecker form, which can be rewritten in matrix-oriented
formulation~\cite{d2019matrix}. This viewpoint avoids the explicit
construction and manipulation of large Kronecker matrices and leads to
efficient formulations for both the continuous semi-discrete problem and the
associated time integrators.

Let \(\Omega=[0,l_x]\times[0,l_y]\) and assume that the diffusion operators
admit tensor-product discretizations. Then
\begin{equation}
\label{eq:A_intro}
\begin{array}{l}
A_u = I_y \otimes T_{11} + T_{12}^T \otimes I_x,\\[0.4em]
A_v = I_y \otimes T_{21} + T_{22}^T \otimes I_x,
\end{array}
\end{equation}
where the matrices \(T_{11},T_{21}\) discretize the second derivative in one
spatial direction and \(T_{12},T_{22}\) the second derivative in the other.
At each time \(t\), we explicitly use the matrices
\(U(t),V(t)\in\mathbb R^{N_x\times N_y}\) containing the same entries as the
vector unknowns \(\mathbf u(t)\) and \(\mathbf v(t)\). Then
\eqref{eq:MOL_intro} can be written as the differential matrix equation
\begin{equation}
\label{eq:MOLmatricial}
\left\{
\begin{array}{lll}
\dot U &=& T_{11}U+UT_{12}+F(U,V),\\[0.3em]
\dot V &=& T_{21}V+VT_{22}+G(U,V),
\end{array}
\right.
\end{equation}
where the entries of \(F(U,V)\) and \(G(U,V)\) are defined entrywise by
\[
F(U,V)_{ij}=f(U_{ij},V_{ij}),\qquad
G(U,V)_{ij}=g(U_{ij},V_{ij}).
\]

The matrix-oriented approach has proved to be an effective alternative to
large-scale vector formulations, drastically reducing computational costs
while retaining structural advantages that are particularly appealing in the
construction of geometric and operator-splitting time
integrators~\cite{hairer2006geometric,blanes2016concise}. Splitting methods
for reaction--diffusion systems have been widely used in applications ranging
from ecological modelling to pattern
formation~\cite{diele2017numerical}. Numerical and algorithmic approaches to
Turing pattern formation in ecological reaction--diffusion models have also
been proposed~\cite{campagna2017semi}. Other time-discretization strategies for
semi-discretized reaction--diffusion PDEs include stable high-order
Runge--Kutta TASE methods~\cite{conte2026general}, as well as
problem-adapted IMEX and trigonometrically fitted methods designed to exploit
specific qualitative features of the underlying dynamics
\cite{d2019adapted}. The present work takes a
complementary viewpoint: we consider first-order splitting integrators
obtained by combining implicit Euler or exact treatment of the diffusive flow
with local reaction solvers of symplectic, adjoint-symplectic, Poisson, and
explicit-variant type~\cite{diele2020geometric}, and ask whether the
corresponding fully discrete maps preserve the Turing region of the underlying
continuous model. Our aim is not merely to derive efficient matrix-oriented
realizations of such methods, but to analyze how their fully discrete dynamics
modifies the onset of diffusion-driven instability.

This question is close to, but distinct from, existing studies of Turing
instability in discrete settings. Coupled map lattices define their own
discrete instability mechanisms, and the corresponding Turing regions are
intrinsic to the discrete model \cite{huang2019exploring}. Reaction--diffusion systems
on networks and, more generally, on discrete topologies provide another
important instance in which the discreteness belongs to the model itself,
rather than to the time integrator \cite{muolo2024turing}. Numerical studies of
reaction--diffusion pattern formation, including IMEX time discretizations,
have also analysed how time stepping affects the growth of unstable modes
\cite{ruuth1995implicit}. Here, however, the question is different: we ask whether the
fully discrete map induced by a time integrator preserves the Turing region of
the underlying continuous reaction--diffusion system. In this sense,
Turing-region preservation is treated as a qualitative property of the
continuous flow.

This perspective places the present study within the broader framework of
\emph{geometric numerical integration}, whose guiding principle is that a good
integrator should reproduce the qualitative features of the continuous flow~\cite{diele2020geometric,blanes2016concise, hairer2006geometric}.
This perspective is in line with the broader development of
structure-preserving discretizations for nonlinear parabolic problems, where
qualitative properties such as positivity, boundedness, entropy stability, or
long-time stability are built into the numerical scheme
\cite{antonietti2025structure}.  For pattern-forming reaction--diffusion systems, the qualitative feature that
matters most is the diffusion-driven instability itself. A faithful integrator
should reproduce, on the retained spatial spectrum, the same separation between
stable and unstable modes predicted by the continuous Turing polynomial.  From this point of view, the central
object is not a single stability threshold but the entire set of parameters
and modes for which the fully discrete map is unstable. Each time integrator
induces its own \emph{discrete Turing region}, which in general differs from
the continuous one: it may anticipate the instability on modes that are
continuously stable, or suppress modes that are continuously unstable. The
main message of this paper is therefore methodological: before a scheme is
used to interpret numerical patterns, its discrete Turing region must be
compared with the continuous one, because consistency alone does not guarantee
that the two coincide at the time steps actually used in long-time
simulations.

The present paper addresses the study of Turing-region preservation in several
directions. First, the analysis is formulated for a general class of
two-species reaction--diffusion systems on rectangular domains, in
matrix-oriented form. Second, it derives fully discrete instability criteria for
entire first-order families of splitting schemes, namely
\(\mathrm{IE}\!-\!\mathcal S\) and \(\mathrm{EX}\!-\!\mathcal S\), which treat the
diffusive flow by implicit Euler or exactly, respectively. Third, the resulting
Jury-based characterization separates the information inherited from the
continuous Turing mechanism from discrete-only effects, making it possible to
identify spurious anticipation, delayed detection, and homogeneous-mode
stability restrictions within a unified framework. Fourth, for the IMEX scheme,
whose first Jury condition reproduces the sign of the continuous Turing
polynomial exactly, this characterization yields an explicit sufficient
condition for Turing-region preservation, in which the role of the second Jury
condition is shown to be confined to the Turing-stable regime.

More precisely, after recalling the continuous and semi-discrete Turing
instability conditions, we develop a modal analysis of the fully discrete
matrix-oriented splitting schemes based on their amplification matrices. This
yields explicit Jury conditions written in terms of a generic modal parameter. We use the Gierer--Meinhardt system~\cite{Gierer1972} model to isolate three distinct numerical mechanisms:
faithful reproduction of a genuine Turing instability, suppression of a
continuous instability by an overly dissipative discrete map, and stable
spurious pattern formation generated only by the fully discrete dynamics.

The paper is organized as follows. In Section~\ref{sec:1} we introduce the
matrix-oriented splitting framework, and in
Section~\ref{sec:first_order_families} we describe the first-order
\(\mathrm{IE}\!-\!\mathcal S\) and \(\mathrm{EX}\!-\!\mathcal S\) families
considered in the sequel. Section~\ref{sec:motivating_examples} recalls the
continuous Turing instability and presents two
motivating Gierer--Meinhardt examples that illustrate the discrete pathologies
addressed in the paper. Sections~\ref{sec:turing}
and~\ref{sec:discrete_turing} develop the semi-discrete and fully
discrete Turing analyses, deriving the Jury conditions for the first-order
splitting families. Section~\ref{sec:numerical_investigation} specializes the
theory to the Gierer--Meinhardt benchmark and discusses the numerical
implications in terms of pattern detection, spurious stability, and stable
spurious pattern formation. Final remarks are collected in
Section~\ref{sec:conclusions}.
\section{Splitting schemes in matrix-oriented formulation}
\label{sec:1}

The matrix-oriented version of all the methods considered here splits
\eqref{eq:MOLmatricial} into the diffusive semiflow
\begin{equation*}
\left\{
\begin{array}{lll}
\dot U &=& T_{11}U+UT_{12},\\
\dot V &=& T_{21}V+VT_{22},
\end{array}
\right.
\end{equation*}
and the reaction semiflow
\begin{equation*}
\left\{
\begin{array}{lll}
\dot U &=& F(U,V),\\
\dot V &=& G(U,V).
\end{array}
\right.
\end{equation*}

\noindent
Let \(t_n=nh_t\), \(n=0,\dots,N_t\), with time step \(h_t>0\). If
\(\varphi_{h_t}^{[D]}\) and \(\varphi_{h_t}^{[R]}\) denote the exact
diffusive and reaction semiflows, then,  starting from \(U^0=U(0)\) and \(V^0=V(0)\),
\[
Y^{n+1}:=\varphi_{h_t}^{[D]}\circ\varphi_{h_t}^{[R]}(Y^n),\quad
Y^n=[U^n,V^n]^T,
\]
and its adjoint
\[
Y^{n+1}:=\varphi_{h_t}^{[R]}\circ\varphi_{h_t}^{[D]}(Y^n),\quad
Y^n=[U^n,V^n]^T,
\]
are first-order approximations of \eqref{eq:MOLmatricial}. Replacing one or
both exact subflows by a first-order approximation still yields a first-order
global procedure. Higher-order  schemes can then be obtained by
composition of methods and their adjoints.

\noindent
For example,   the first-order
IMEX method combines the explicit Euler method for the reaction semiflow with
the implicit Euler approximation of the diffusive semiflow:
\begin{equation*}
\begin{array}{ll}
U^{n+1}-U^{n_1}=h_t(T_{11}U^{n+1}+U^{n+1}T_{12}),
& U^{n_1}=U^n+h_tF(U^n,V^n),\\[0.8em]
V^{n+1}-V^{n_1}=h_t(T_{21}V^{n+1}+V^{n+1}T_{22}),
& V^{n_1}=V^n+h_tG(U^n,V^n).
\end{array}
\end{equation*}
Combining instead the explicit Euler method for the reaction semiflow with
the exact solution of the diffusive semiflow, we obtain the method
\(\mathrm{EX}\!-\!\mathrm{EE}\):
\begin{equation*}
\begin{array}{ll}
\displaystyle U^{n+1}=e^{h_tT_{11}} \, U^{n_1} \, e^{h_tT_{12}},
&\quad U^{n_1}-U^n=h_t\,F(U^{n},V^{n}), \\[0.6em]
\displaystyle V^{n+1}=e^{h_tT_{21}} \, V^{n_1} \, e^{h_tT_{22}},
&\quad V^{n_1}-V^n=h_t\,G(U^{n},V^{n}).
\end{array}
\end{equation*}

\section{First-order symplectic  families}
\label{sec:first_order_families}

In this section, we focus on methods based on reaction-diffusion splitting
with either implicit Euler or exact treatment of the diffusive flow. For
the reaction step we consider the first-order local integrators
\(\mathcal S \in \mathfrak S\), where $\mathfrak S$ is the set of first-order integrators listed in Table~\ref{tab:reaction_variants}:
Symplectic Euler (SE), Adjoint of Symplectic Euler (ASE), Poisson Euler
(PE), Adjoint of Poisson Euler (APE), together with their explicit variants
denoted by the prefix \(\mathrm{EV}\). Since the reaction semiflow is local,
all the following formulas are understood component-wise when applied to the
matrices \(U\) and \(V\).

\begin{table}[htbp]
\centering
\small
\renewcommand{\arraystretch}{1.5}
\begin{tabular}{p{0.14\textwidth}p{0.37\textwidth}p{0.37\textwidth}}
\hline
Method & \(u\)-update & \(v\)-update \\
\hline
\noalign{\vskip 2pt}
SE &
\(\displaystyle \frac{u^{n_1}-u^n}{h_t}=f(u^{n_1},v^n)\) &
\(\displaystyle \frac{v^{n_1}-v^n}{h_t}=g(u^{n_1},v^n)\) \\[1.2ex]
EVSE &
\(\displaystyle \frac{u^{n_1}-u^n}{h_t}=f(u^n,v^n)\) &
\(\displaystyle \frac{v^{n_1}-v^n}{h_t}=g(u^{n_1},v^n)\) \\[1.2ex]
ASE &
\(\displaystyle \frac{u^{n_1}-u^n}{h_t}=f(u^n,v^{n_1})\) &
\(\displaystyle \frac{v^{n_1}-v^n}{h_t}=g(u^n,v^{n_1})\) \\[1.2ex]
EVASE &
\(\displaystyle \frac{u^{n_1}-u^n}{h_t}=f(u^n,v^{n_1})\) &
\(\displaystyle \frac{v^{n_1}-v^n}{h_t}=g(u^n,v^n)\) \\[1.2ex]
PE &
\(\displaystyle u^{n_1}=u^n\exp\!\left(h_t\frac{f(u^{n_1},v^n)}{u^n}\right)\) &
\(\displaystyle v^{n_1}=v^n\exp\!\left(h_t\frac{g(u^{n_1},v^n)}{v^n}\right)\) \\[1.2ex]
EVPE &
\(\displaystyle u^{n_1}=u^n\exp\!\left(h_t\frac{f(u^n,v^n)}{u^n}\right)\) &
\(\displaystyle v^{n_1}=v^n\exp\!\left(h_t\frac{g(u^{n_1},v^n)}{v^n}\right)\) \\[1.2ex]
APE &
\(\displaystyle u^{n_1}=u^n\exp\!\left(h_t\frac{f(u^n,v^{n_1})}{u^n}\right)\) &
\(\displaystyle v^{n_1}=v^n\exp\!\left(h_t\frac{g(u^n,v^{n_1})}{v^{n_1}}\right)\) \\[1.2ex]
EVAPE &
\(\displaystyle u^{n_1}=u^n\exp\!\left(h_t\frac{f(u^n,v^{n_1})}{u^n}\right)\) &
\(\displaystyle v^{n_1}=v^n\exp\!\left(h_t\frac{g(u^n,v^n)}{v^n}\right)\) \\
\noalign{\vskip 2pt}
\hline
\renewcommand{\arraystretch}{0.5}
\end{tabular}
\caption{Local first-order approximations of the reaction semiflow.
When used in the matrix-oriented framework, the formulas are applied
component-wise to all grid nodes.}
\label{tab:reaction_variants}
\end{table}

\subsection{The family of \texorpdfstring{$\mathrm{IE}\!-\!\mathcal S$}{IE-R}
schemes}

Let \(Y^n=[U^n,V^n]^T\), and let
\[
(U^{n_1},V^{n_1})=\Phi_{h_t}^{[R,\mathcal S]}(U^n,V^n),\quad
\mathcal S\in\mathfrak S,
\]
where \(\Phi_{h_t}^{[R,\mathcal S]}\) denotes the componentwise reaction map
associated with one of the methods listed in
Table~\ref{tab:reaction_variants}. Then the first-order matrix-oriented
scheme with implicit Euler diffusion and reaction solver \(\mathcal S\) is
\begin{equation*}
\mathrm{IE}\!-\!\mathcal S:\quad
Y^{n+1}=\Phi_{h_t}^{[D,\mathrm{IE}]}\circ\Phi_{h_t}^{[R,\mathcal S]}(Y^n),
\end{equation*}
where the diffusive part is always treated through the pair of Sylvester
equations
\begin{equation*}
\Phi_{h_t}^{[D,\mathrm{IE}]}: \quad \begin{array}{l}
U^{n+1}-U^{n_1}=h_t(T_{11}U^{n+1}+U^{n+1}T_{12}),\\[0.4em]
V^{n+1}-V^{n_1}=h_t(T_{21}V^{n+1}+V^{n+1}T_{22}).
\end{array}
\end{equation*}

\paragraph{Example: \texorpdfstring{$\mathrm{IE-ASE}$}{IEASE}.}
\begin{equation*}
\begin{array}{ll}
U^{n+1}-U^{n_1}=h_t(T_{11}U^{n+1}+U^{n+1}T_{12}),
& U^{n_1}-U^n=h_tF(U^n,V^{n_1}),\\[0.8em]
V^{n+1}-V^{n_1}=h_t(T_{21}V^{n+1}+V^{n+1}T_{22}),
& V^{n_1}-V^n=h_tG(U^n,V^{n_1}).
\end{array}
\end{equation*}

\paragraph{Example: \texorpdfstring{$\mathrm{IE-SE}$}{IESE}.}
\begin{equation*}
\begin{array}{ll}
U^{n+1}-U^{n_1}=h_t(T_{11}U^{n+1}+U^{n+1}T_{12}),
& U^{n_1}-U^n=h_tF(U^{n_1},V^n),\\[0.8em]
V^{n+1}-V^{n_1}=h_t(T_{21}V^{n+1}+V^{n+1}T_{22}),
& V^{n_1}-V^n=h_tG(U^{n_1},V^n).
\end{array}
\end{equation*}

\subsection{The family of \texorpdfstring{$\mathrm{EX}\!-\!\mathcal S$}{EX-R}
schemes}

For any \(\mathcal S\in\mathfrak S\), we define
\begin{equation*}
\mathrm{EX}\!-\!\mathcal S:\quad
Y^{n+1}=\varphi_{h_t}^{[D]}\circ\Phi_{h_t}^{[R,\mathcal S]}(Y^n),
\end{equation*}
where
\(
\varphi_{h_t}^{[D]}(U,V)=
\bigl(e^{h_tT_{11}}Ue^{h_tT_{12}},e^{h_tT_{21}}Ve^{h_tT_{22}}\bigr).
\)

\paragraph{Example: \texorpdfstring{$\mathrm{EX-ASE}$}{EXASE}.}
\begin{equation*}
\begin{array}{ll}
\displaystyle U^{n+1}=e^{h_tT_{11}} \, U^{n_1} \, e^{h_tT_{12}},
&\quad U^{n_1}-U^n=h_t\,F(U^{n},V^{n_1}), \\[0.6em]
\displaystyle V^{n+1}=e^{h_tT_{21}} \, V^{n_1} \, e^{h_tT_{22}},
&\quad V^{n_1}-V^n=h_t\,G(U^{n},V^{n_1}).
\end{array}
\end{equation*}
For \(\mathrm{EX}\!-\!\mathcal S\) schemes, when the problem dimension is
large, the exact computation of matrix exponentials may be expensive, and
approximate low-rank strategies can be considered~\cite{mena2018numerical}. However, such
approximations may bias the selection of unstable modes during transient
pattern formation.

\section{Continuous Turing Instability}
\label{sec:motivating_examples}
We consider the prototypical form of \eqref{eq:model_intro}, namely the reaction--diffusion system with constant diffusions $D_u, D_v >0$:
\begin{equation}
\label{eq:general_hu}
\begin{cases}
u_t = D_u \Delta u + f(u,v),\\
v_t = D_v \Delta v + g(u,v),
\end{cases}
\end{equation}
where $\Delta$ denotes the Laplacian, together with a spatially homogeneous equilibrium \((u^*,v^*)\) whose Jacobian is
\begin{equation}\label{eq:Jstar}
J^*=
\begin{pmatrix}f_u & f_v\\ g_u & g_v\end{pmatrix}_{\!(u^*,v^*)}.
\end{equation}
The equilibrium is assumed stable in the absence of diffusion (kinetic stability):
\begin{equation}
\label{eq:conditions_lin}
\tau^*:=\operatorname{tr}(J^*)=f_u+g_v<0,\qquad
\delta^*:=\det(J^*)=f_ug_v-f_vg_u>0.
\end{equation}
Linearizing \eqref{eq:general_hu} about \((u^*,v^*)\) and expanding the perturbation over the eigenfunctions of the Laplacian, $-\Delta\varphi=\mu\varphi$ with $\mu\ge 0$, each mode evolves independently according to
\begin{equation}
\label{eq:continuous_modal}
\mathbf w_t=J_\mu\mathbf w,\qquad
J_\mu=J^*-\mu\,D,\qquad
D=\begin{pmatrix}D_u & 0\\ 0 & D_v\end{pmatrix},
\end{equation}
with characteristic polynomial
\begin{equation}
\label{eq:charpoly_cont}
\lambda^2+\bigl((D_u+D_v)\mu-(f_u+g_v)\bigr)\lambda+h(\mu)=0,
\end{equation}
where
\begin{equation}
\label{eq:h_cont}
h(\mu)=D_uD_v\,\mu^2-(D_ug_v+D_vf_u)\,\mu+(f_ug_v-f_vg_u).
\end{equation}
The mode $\mu=0$ is the spatially homogeneous one, stable by \eqref{eq:conditions_lin}.

Under \eqref{eq:conditions_lin}, a Turing instability occurs if and only if $h(\mu)<0$ for some $\mu>0$. Since $D_uD_v>0$ and $\delta^*>0$, the convex parabola $h$ takes negative values for some $\mu>0$ if and only if
\begin{equation}
\label{eq:algebraic_continum}
m:=D_ug_v+D_vf_u>0
\qquad\text{and}\qquad
m^2>4\,D_uD_v\,\delta^* .
\end{equation}

\begin{remark}
\label{rem:opposite_signs_eq}
By \eqref{eq:algebraic_continum}, Turing instability requires $m=D_u g_v + D_v f_u >0$.
Hence, $f_u$ and $g_v$ must have opposite signs, namely 
\begin{equation}
\label{eq:cond_fugv_neg}
    f_u\, g_v<0 .
\end{equation}
\end{remark}

\subsection{The Gierer--Meinhardt  benchmark model} We focus on the Gierer--Meinhardt activator--inhibitor system
\begin{equation}
\label{eq:gm_model_motivating}
\begin{cases}
u_t = D_u\,\Delta u + \gamma\left(a+\dfrac{u^2}{v}-c\,u\right),\\[6pt]
v_t = D_v\,\Delta v + \gamma\left(u^2-b\,v\right),
\end{cases}
\end{equation}
with homogeneous Neumann boundary conditions. The positive spatially
homogeneous equilibrium is
\begin{equation}
\label{eq:gm_equilibrium_motivating}
P_e=(u_e,v_e)
=
\left(
\dfrac{a+b}{c},
\dfrac{(a+b)^2}{c^2b}
\right)
\end{equation}
and the Jacobian of the reaction kinetics at \(P_e\) is
\begin{equation}
\label{eq:gm_jacobian_motivating}
J^*
=
\begin{pmatrix}
\displaystyle \gamma\,\dfrac{c(b-a)}{a+b}
&
\displaystyle -\gamma\,\dfrac{c^2b^2}{(a+b)^2}
\\[8pt]
\displaystyle \dfrac{2\gamma(a+b)}{c}
&
-\gamma b
\end{pmatrix}.
\end{equation}


We adopt the tensor-product finite-difference semi-discretization with
homogeneous Neumann boundary conditions described in~\cite{d2019matrix}.
More precisely, we consider the second-order ECDF discretization on
$\Omega=[0,\ell_x]\times[0,\ell_y]$, with $N_x$ and $N_y$ interior
nodes and mesh sizes
\[
  h_x=\frac{\ell_x}{N_x+1},\qquad h_y=\frac{\ell_y}{N_y+1}.
\]
The one-dimensional second-derivative matrices are
\[
  T_x = \operatorname{tridiag}(1,-2,1)+B_x,\qquad
  T_y = \operatorname{tridiag}(1,-2,1)+B_y,
\]
where $B_x$, $B_y$ are the ECDF boundary correction matrices enforcing
the Neumann conditions.  The resulting matrix-oriented semi-discretization
of system~\eqref{eq:model_intro} is
\begin{equation*}
\left\{
\begin{array}{l}
\dot U
=
D_u\!\left(\dfrac{T_x}{h_x^2}U+\dfrac{U\,T_y^T}{h_y^2}\right)+F(U,V),
\\[6pt]
\dot V
=
D_v\!\left(\dfrac{T_x}{h_x^2}V+\dfrac{V\,T_y^T}{h_y^2}\right)+G(U,V),
\end{array}
\right.
\end{equation*}
where $F(U,V)_{i,j}= \gamma \left(a+\dfrac{U_{i,j}^2}{V_{i,j}}-c\,U_{i,j}\right)$ and $G(U,V)_{i,j}=\gamma\left(U_{i,j}^2-b\,V_{i,j}\right)$.

\paragraph{Turing patterns: vectorial vs. matrix-oriented  application.} To assess the computational advantages of the matrix-oriented formulation, we compare the vectorial and matricial implementations of three representative schemes, namely IMEX, IE--SE and IE--EVSE, on the Gierer--Meinhardt problem with parameters
$$D_u=1,\qquad D_v=50,\qquad a=1,\qquad b=2,\qquad c=4,\qquad \gamma=10.$$
The simulations are performed on the same spatial domain $[0,10]$ with $n_x = n_y = 100
$ interior nodes and from the same initial random perturbation of the homogeneous equilibrium.
Since the objective is to compare the computational effort required to reach the stationary patterned state, we employ a stopping criterion based on the increment of the unknown $U$, that is the Frobenius norm of the difference between two consecutive numerical solutions. Specifically, the time integration is terminated as soon as
$\|U^{k+1}-U^{k}\|_F < 10^{-10},$
where $\|\cdot\|_F$ denotes the Frobenius norm. This criterion is commonly adopted in the reaction--diffusion literature \cite{AMS2023,Alla_pDMD,monti2025socpatterns}, as an indicator that the numerical solution has approached a stationary state.

\begin{figure}[htbp]
    \centering
\includegraphics[width=0.33\linewidth]{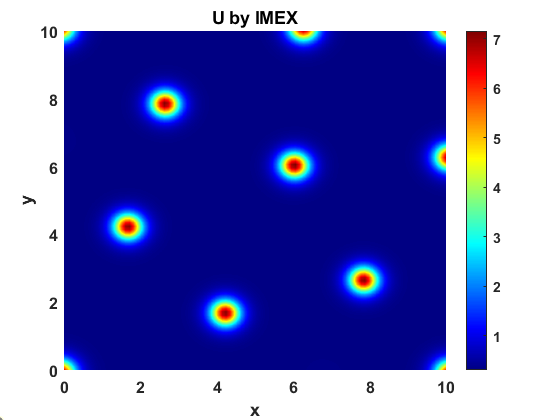}
\includegraphics[width=0.33\linewidth]{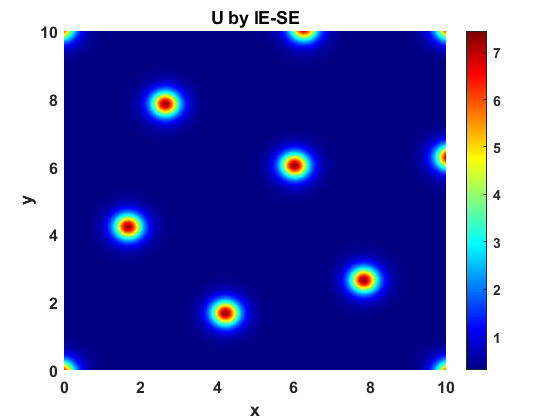}
\includegraphics[width=0.33\linewidth]{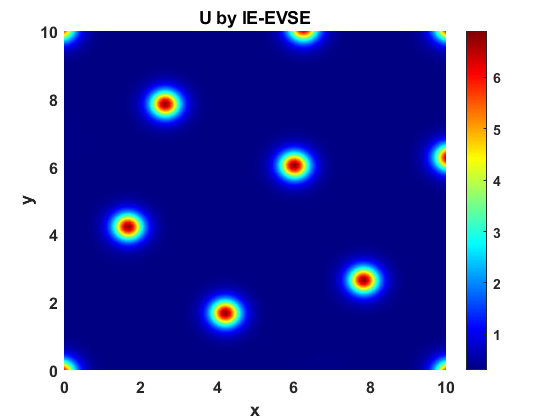}
    \caption{Stationary spot-like patterns obtained with the matrix-oriented formulation of the IMEX,
IE--SE and IE--EVSE schemes for the Gierer--Meinhardt model. The solution is shown at the final time $T$ determined by the stopping criterion and reported in Table \ref{tab:cfr_matrixvec}.}
    \label{fig:pattern_matrix}
\end{figure}

In Figure \ref{fig:pattern_matrix}, we report the spatial patterns obtained with the IMEX, IE--SE and IE--EVSE schemes. Despite the different
time discretizations, all methods converge to the same stable spot pattern. This confirms that the three schemes capture the same qualitative
Turing dynamics and therefore provide a suitable benchmark for comparing
their computational efficiency.

To quantify the computational gain, Table~\ref{tab:cfr_matrixvec} reports, for each scheme, the final simulation time $T$ (determined by the stopping criterion), the CPU time required to reach stationarity, and the corresponding speed-up factor with respect to the vector formulation. The comparison highlights the efficiency of the matrix-oriented implementation, which avoids the explicit assembly of large Kronecker-product matrices and exploits the intrinsic tensor-product structure of the discretized diffusion operators.

\begin{table}[htbp]
\centering
\begin{tabular}{lccc}
\hline
Method & $T$ & CPU time (s) & Speed-up factor \\
\hline
IMEX & $2.6386 \times 10^3$ & $279.42$ & $251.09$\\
IE--SE & $2.7476 \times 10^3$ & $515.26$ & $142.37$\\
IE--EVSE & $2.6282 \times 10^3$ & $281.34$ & $248.15$\\
\hline
\renewcommand{\arraystretch}{0.5}
\end{tabular}
\caption{Final simulation time, CPU time and speed-up factor for the matrix-oriented implementations of the considered schemes.}
\label{tab:cfr_matrixvec}
\end{table}

\subsection{Motivating examples}

The two examples below illustrate two opposite fully discrete pathologies.
In the first one, the continuous model is Turing stable, yet IMEX generates a
stable spurious numerical pattern. In the second one, the system lies in a Turing-unstable regime in the continuous setting, but the IE--ASE family spuriously suppresses the pattern.

\subsubsection{Example 1: A stable spurious IMEX pattern in a stable regime}
\label{subsec:motivating_imex_spurious}

We first consider a parameter set for which the continuous
Gierer--Meinhardt system is kinetically stable and Turing stable. The
parameters are
\begin{equation}
\label{eq:imex_spurious_parameters}
D_u=1,
\qquad
D_v=12,
\qquad
\gamma=10,\qquad a=2,
\qquad
b=0.15,
\qquad
c=0.4 .
\end{equation}
For these values the homogeneous equilibrium is
$(u_e,v_e)=(5.375,192.604)$, with
\[
J^*=\begin{pmatrix} -3.4419 & -0.0078\\ 107.5 & -1.5 \end{pmatrix},
\qquad
\tau^*=-4.9419<0,
\qquad
\delta^*=6>0,
\]
so the equilibrium is kinetically stable. Moreover
\[
D_ug_v+D_vf_u=-42.80<0,
\qquad\text{hence}\qquad
h(\mu)\ge\delta^*=6>0
\quad\text{for all } \mu\ge 0,
\]
and the continuous problem has no diffusion-driven instability.

Nevertheless, for the IMEX scheme, namely IE--EE, with
\(
h_t=h_t^{(1)}=0.67,
\)
which is below the IMEX homogeneous-mode threshold
$\bar h_{\mathrm{hom}}^{\mathrm{EE}}=0.7154$, the fully discrete dynamics
exhibits a spurious non-homogeneous destabilization
($\max_\mu\rho(\mu)\approx1.04>1$), and the nonlinear discrete solution evolves
toward a bounded patterned numerical state. By contrast, the SE/PE-type schemes
damp the same perturbation and return to the homogeneous equilibrium.

Figure~\ref{fig:ex1_imex_spurious_combined} summarizes the phenomenon. The
left panel shows the final spatial profiles. IMEX produces a stable
non-homogeneous numerical pattern, while the SE/PE-type schemes do not. 
The right panel reports the evolution of the spatial mean of $U$. The SE-type schemes preserve the homogeneous equilibrium, yielding a constant spatial mean throughout the simulation. In contrast, IE--EE (IMEX) undergoes a destabilization and converges to a nonhomogeneous steady state corresponding to a stable pattern. The EX--EE scheme exhibits an unbounded growth of the spatial mean, reflecting the onset of an unstable transient pattern. The solution shown in the left panel corresponds to the last time step before the method breaks down, as the numerical solution loses positivity.

\begin{figure}[htbp]
\centering
\begin{subfigure}[t]{0.47\textwidth}
    \centering \includegraphics[width=6.5cm]{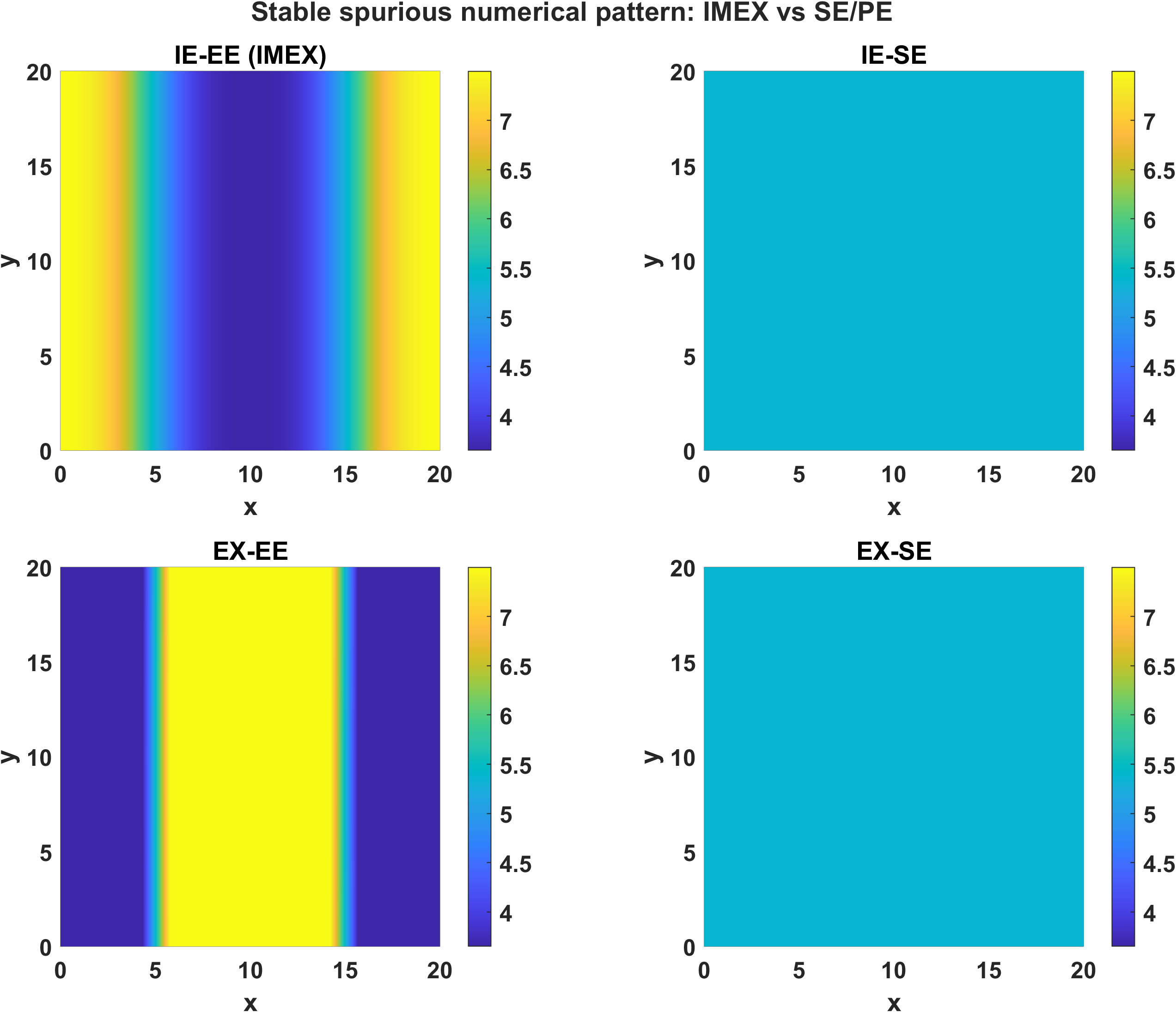}
 \caption{Final spatial profiles.}
  \label{fig:ex1_imex_spurious_patterns}
\end{subfigure}
\hfill
\begin{subfigure}[t]{0.47\textwidth}
    \centering
    \includegraphics[width=\textwidth]{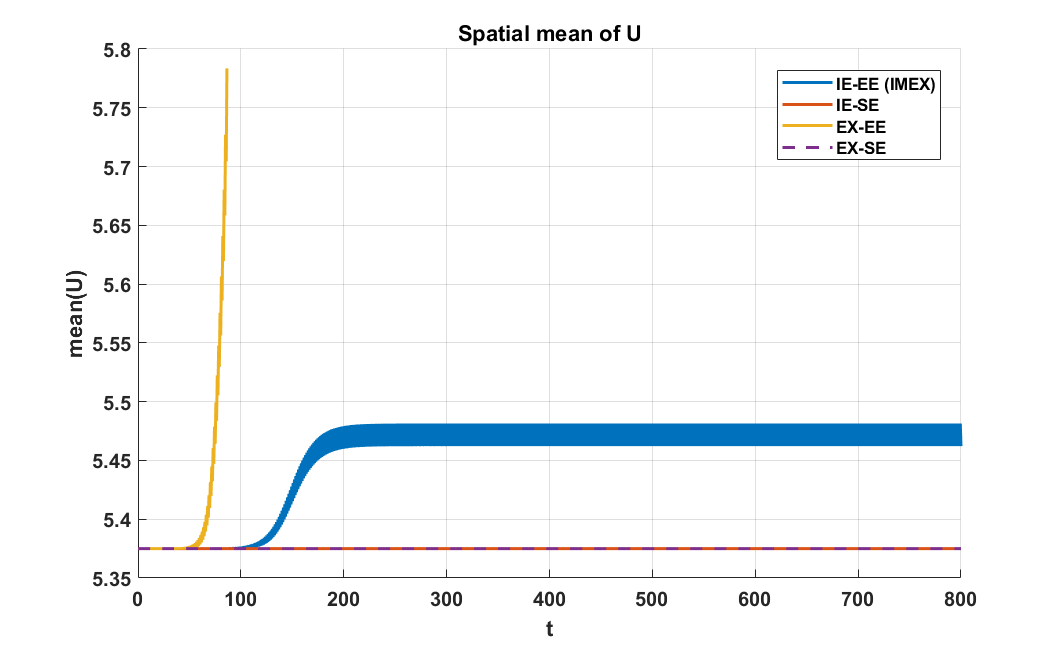}
    \caption{Spatial mean of the solution.}
    \label{fig:ex1_imex_spurious_rms}
\end{subfigure}
\caption{
Stable spurious numerical pattern generated by IMEX in a Turing-stable regime for the Gierer--Meinhardt model. The continuous model satisfies \(h(\mu)>0\) for all
\(\mu\ge0\), so no Turing instability is present. Nevertheless,
IE--EE (IMEX) amplifies a non-homogeneous discrete mode and saturates to a
bounded spatially patterned numerical state. In contrast, the SE/PE-type
schemes damp the same perturbation and return to the homogeneous equilibrium.
}
\label{fig:ex1_imex_spurious_combined}
\end{figure}


\subsubsection{Example 2: {Spurious stability of IE--ASE in a  Turing regime.}}
\label{subsec:motivating_iease_spurious_stability}

We now choose parameters for which the continuous Gierer--Meinhardt model exhibits a Turing instability. The parameters are
\begin{equation}
\label{eq:iease_spurious_parameters}
D_u=1,
\qquad
D_v=160,
\qquad
\gamma=5, \qquad a=0.6933,
\qquad
b=2,
\qquad
c=1.
\end{equation}
For these values the equilibrium is kinetically stable, while
\(
h(\mu)<0
\)
on a non-empty interval of positive modes. Hence the continuous model
predicts a diffusion-driven instability.

At the time step
\(
h_t=h_t^{(2)}=0.05,
\)
the IMEX scheme detects the continuous Turing instability and produces the
expected non-homogeneous pattern. However, the IE--ASE scheme, which
corresponds to the implicit-symplectic (IMSP) scheme introduced
in~\cite{diele2017numerical}, behaves differently: instead of
amplifying the unstable modes, it suppresses them and drives the solution
back toward the homogeneous equilibrium. Figure~\ref{fig:ex2_iease_spurious_stability_combined} summarizes the
comparison: the left panel shows that IMEX produces the expected Turing
pattern, whereas IE--ASE and EX--ASE suppress it. 

The right panel confirms the same mechanism through the evolution of the spatial mean of $U$. The IE–ASE and EX–ASE schemes maintain the homogeneous equilibrium, resulting in a constant spatial mean over time. On the contrary, the IE–EE (IMEX) method departs from the initial equilibrium and settles to a lower constant value, indicating convergence toward a different steady configuration. The EX–EE scheme instead shows a rapid transient followed by a stabilization at an intermediate level, highlighting a distinct dynamical behaviour compared to the other methods. 

\begin{figure}[htbp]
\centering
\begin{subfigure}[t]{0.47\textwidth}
    \centering
    \includegraphics[width=6.5cm]{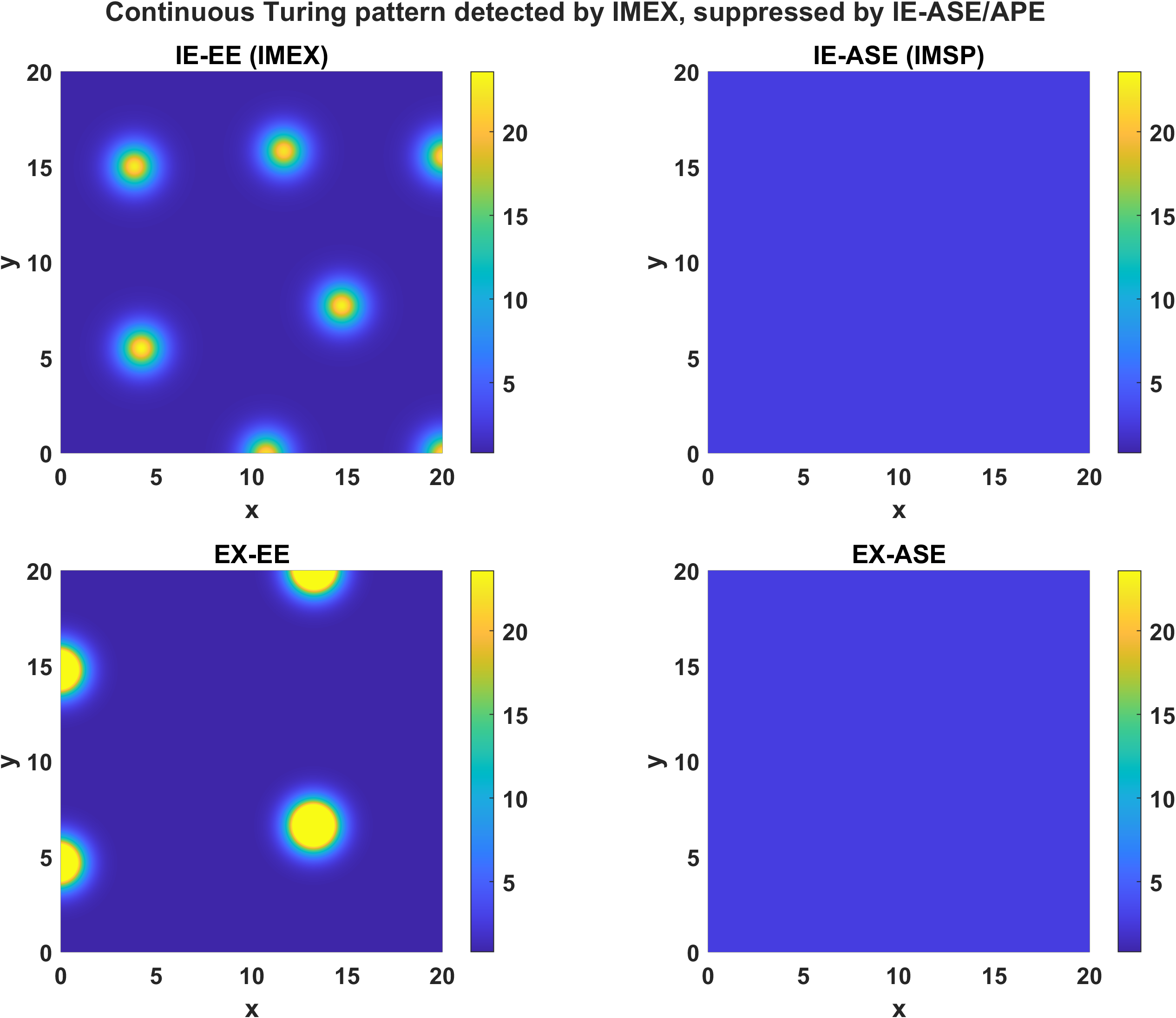}
    \caption{Final spatial profiles.}
\label{fig:ex2_iease_spurious_stability_patterns}
\end{subfigure}
\hfill
\begin{subfigure}[t]{0.47\textwidth}
    \centering
\includegraphics[width=1\textwidth]{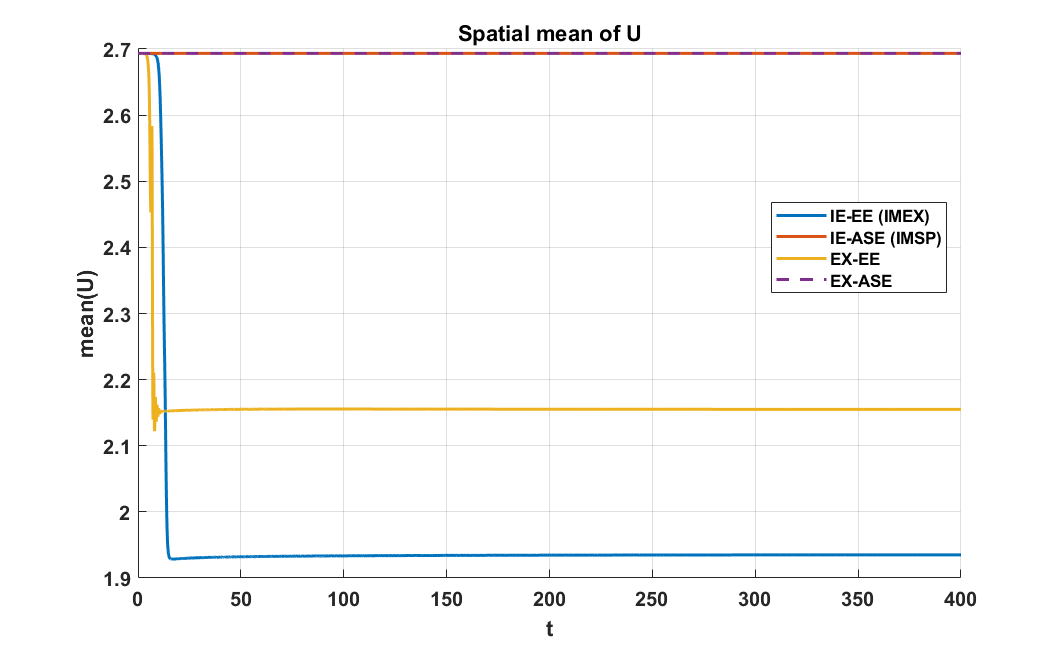}
    \caption{Spatial mean of the solution.}
    \label{fig:ex2_iease_spurious_stability_rms}
\end{subfigure}
\caption{
Spurious stability of the ASE/APE family in a continuous Turing
regime. The continuous Gierer--Meinhardt model is Turing unstable, and IMEX
detects the instability by producing a patterned state. In contrast,
IE--ASE/APE and EX--ASE/APE suppress the unstable modes and drive the
solution back toward the homogeneous equilibrium, as confirmed by the constant spatial mean shown in panel (b). 
}
\label{fig:ex2_iease_spurious_stability_combined}
\end{figure}

Taken together, the two examples show that the time discretization may
distort the continuous Turing mechanism in two opposite directions. IMEX may
create a stable patterned numerical state in a Turing-stable regime, while
the ASE/APE family may suppress Turing growth in an
unstable regime predicted by the continuous model. This motivates the modal Jury analysis developed in the next
sections.

\section{ Discrete Turing instability of one-step integrators}
\label{sec:turing}
A spatial semi-discretization of \eqref{eq:model_intro} on a
tensor-product grid leads to the system in Eq.~\eqref{eq:MOL_intro},
where $A_u, A_v$ are defined in \eqref{eq:A_intro}, with $T_{11} = D_u T_x$, $T_{12} = D_u T_y$, $T_{21} = D_v T_x$, $T_{22} = D_v T_y$. By the tensor-product structure in \eqref{eq:A_intro}, $A_u$ and $A_v$ are simultaneously diagonalized by the common eigenbasis of the one-dimensional second-derivative matrices $T_x, T_y$. If $\lambda_i^x$ and $\lambda_j^y$ denote the eigenvalues of $T_x$ and $T_y$, the associated diffusion eigenvalues are the grid modes
\begin{equation}
\label{eq:mu_modes}
\mu_{ij}=-(\lambda_i^x+\lambda_j^y)\geq0,
\end{equation}
which we generically denote by $\mu\ge0$. Projecting the linearized dynamics onto each such mode decouples it into independent two-dimensional blocks indexed by $\mu$. Applying a one-step
method with time step $h_t$ to this system replaces the continuous semiflow by
a discrete map
\begin{equation*}
z^{n+1} = \Phi(z^n; h_t, \mu),
\end{equation*}
where $\Phi$ is the numerical flow operator generated by the one-step scheme.
The stability of the homogeneous equilibrium under the discrete scheme is controlled
by the spectral radius of the two-dimensional amplification matrix
\begin{equation*}
G_{h_t}(\mu) := \bigl.\mathrm{D}\Phi\bigr|_{(u^*,v^*)},
\end{equation*}
the $2\times2$ block obtained by projecting the Jacobian of the discrete map, evaluated at the equilibrium, onto the eigenspace of mode $\mu$. The
equilibrium is linearly stable on mode $\mu$ if and only if the spectral radius satisfies
$\rho(G_{h_t}(\mu)) < 1$.
\begin{proposition}[Jury stability criterion \cite{periyasamy2015determinant}, \cite{jury1978stability}]
\label{prop:jury_turing}
Let \(\tau_{h_t}(\mu)=\operatorname{tr}(G_{h_t}(\mu))\) and
\(\delta_{h_t}(\mu)=\det(G_{h_t}(\mu))\). Then
\(\rho(G_{h_t}(\mu))<1\) if and only if
\begin{equation}
\label{eq:Jury_cond}
\mathcal J_1(\mu)>0,\quad\mathcal J_2(\mu)>0,\quad\mathcal J_3(\mu)>0,
\end{equation}
where
\begin{equation}
\label{eq:Jury}
\mathcal J_1(\mu)=1-\tau_{h_t}(\mu)+\delta_{h_t}(\mu),\quad
\mathcal J_2(\mu)=1+\tau_{h_t}(\mu)+\delta_{h_t}(\mu),\quad
\mathcal J_3(\mu)=1-\delta_{h_t}(\mu),\qquad \mu \geq 0.
\end{equation}
\end{proposition}

In the rest of the paper, we refer to $\mathcal J_i(\mu)$, $i=1,2,3$, defined in \eqref{eq:Jury}, as the Jury quantities, and to the inequalities $\mathcal J_i(\mu)>0$ \eqref{eq:Jury_cond} as the corresponding Jury conditions.

\begin{remark}
\label{rem:why_jury}
Linear stability of the $2\times2$ map $G_{h_t}(\mu)$ could equivalently be
expressed through the Schur conditions directly on its trace and determinant,
namely $|\tau_{h_t}(\mu)|<1+\delta_{h_t}(\mu)$ and $\delta_{h_t}(\mu)<1$.
We deliberately adopt the Jury form \eqref{eq:Jury} because the three quantities
$\mathcal J_1,\mathcal J_2,\mathcal J_3$ separate the stability requirement into
mechanisms with distinct physical content. As shown in
Proposition~\ref{prop:continuous_interpretation}, the first quantity
$\mathcal J_1$ carries, in the limit $h_t\to0$, exactly the sign of the
continuous Turing polynomial $h(\mu)$: its loss of positivity is the fully discrete
counterpart of the diffusion-driven instability of the underlying PDE.
The second quantity $\mathcal J_2$, on the contrary, has no continuous
counterpart and encodes a purely discrete stability constraint introduced by
the time integrator. A violation of $\mathcal J_2 >0$ therefore signals an
instability that is spurious with respect to the continuous problem. The Jury
formulation thus lets us isolate the Turing mechanism inherited from the continuous model,
encoded in $\mathcal J_1$, from purely discrete artefacts, encoded in
$\mathcal J_2$, a distinction that the raw trace--determinant inequalities
would not make transparent.
\end{remark}
\begin{definition}[Fully discrete Turing instability]
\label{def:fully_discrete_turing}
A one-step integrator exhibits a fully discrete Turing instability if
$\rho(G_{h_t}(0))<1$
and there exists \(\mu>0\) such that
$\rho(G_{h_t}(\mu))>1$.
Equivalently, the Jury conditions are satisfied at the homogeneous mode
\(\mu=0\), while for some \(\mu>0\) at least one of the Jury conditions is
strictly violated.
\end{definition}

\section{Discrete Turing instability of the IE--$\mathcal S$ and EX--$\mathcal S$ families}
\label{sec:discrete_turing}
%
Before providing  conditions for  fully-discrete Turing instability for the 
 $\mathrm{IE}\!-\!\mathcal S$ and $\mathrm{EX}\!-\!\mathcal S$ families of 
Section~\ref{sec:first_order_families}, we first derive their explicit formulas for $G_{h_t}(\mu)$.

\begin{lemma}
\label{lem:general_modal_matrix}
Consider the modal system \eqref{eq:continuous_modal} evaluated on the discrete mode $\mu=\mu_{ij}\ge0$ of \eqref{eq:mu_modes}.
 Then the one-step amplification for the 
 $\mathrm{IE}\!-\!\mathcal S$ and $\mathrm{EX}\!-\!\mathcal S$ schemes with \(\mathcal S \in \mathfrak S\) is given by 
\begin{equation*}
G_{h_t}^{\mathcal D-\mathcal S}(\mu)=B_{h_t}^{\mathcal D}(\mu)\,A_{h_t}^{\mathcal S},
\end{equation*}
where $A_{h_t}^{\mathcal S}$ denotes the Jacobian of the linearized
reaction substep at \((u^*,v^*)\) and  $B_{h_t}^{\mathcal D}(\mu)=\operatorname{diag}(b_u^{\mathcal D}(\mu),b_v^{\mathcal D}(\mu))$ where
\begin{equation}
\label{eq:diffusion_factors}
(b_u^{\mathcal D}, b_v^{\mathcal D}) = 
\begin{cases}
(b_u^{\mathrm{IE}}, b_v^{\mathrm {IE}}) := \displaystyle\left(\frac{1}{1+h_t D_u \mu}, 
\frac{1}{1+h_t D_v \mu}\right) 
& \text{for }\mathrm{IE}\!-\!\mathcal{S}\text{ schemes},\\[8pt]
(b_u^{\mathrm{EX}}, b_v^{\mathrm{EX}}) := \left(e^{-h_t D_u \mu}, e^{-h_t D_v \mu}\right) 
& \text{for }\mathrm{EX}\!-\!\mathcal{S}\text{ schemes}.
\end{cases}
\end{equation}
\end{lemma}


The Jacobians of the linearized reaction maps
\(A_{h_t}^{\mathcal S}\), with \(\mathcal S\in\mathfrak S\), are summarized
in  Table \ref{tab:reaction_jacobians}.
\begin{table}[tbp]
\centering
\renewcommand{\arraystretch}{2.5}
\begin{tabular}{c|p{0.72\textwidth}}
\hline
Scheme \(\mathcal S\) & Jacobian \(A_{h_t}^{\mathcal S}\) \\
\hline
EE
&
\(\displaystyle
A_{h_t}^{\mathrm{EE}}=\begin{pmatrix}
1+h_tf_u
&
h_tf_v
\\[6pt]
h_tg_u
&
1+h_tg_v
\end{pmatrix}
\)
\\[0.3cm]
\hline
ASE
&
\(\displaystyle
A_{h_t}^{\mathrm{ASE}}
=
\begin{pmatrix}
1+h_tf_u+\dfrac{h_t^2f_vg_u}{s_g}
&
\dfrac{h_tf_v}{s_g}
\\[6pt]
\dfrac{h_tg_u}{s_g}
&
\dfrac{1}{s_g}
\end{pmatrix}
\)
\\[6pt]
\hline
SE
&
\(\displaystyle
A_{h_t}^{\mathrm{SE}}
=
\begin{pmatrix}
\dfrac{1}{s_f}
&
\dfrac{h_tf_v}{s_f}
\\[6pt]
\dfrac{h_tg_u}{s_f}
&
1+h_tg_v+\dfrac{h_t^2g_uf_v}{s_f}
\end{pmatrix}
\)
\\[6pt]
\hline
EVSE
&
\(\displaystyle
A_{h_t}^{\mathrm{EVSE}}
=
\begin{pmatrix}
1+h_tf_u
&
h_tf_v
\\[4pt]
h_tg_u(1+h_tf_u)
&
1+h_tg_v+h_t^2g_uf_v
\end{pmatrix}
\)
\\
\hline
EVASE
&
\(\displaystyle
A_{h_t}^{\mathrm{EVASE}}
=
\begin{pmatrix}
1+h_tf_u+h_t^2f_vg_u
&
h_tf_v(1+h_tg_v)
\\[4pt]
h_tg_u
&
1+h_tg_v
\end{pmatrix}
\)
\\
\hline 
Poisson-type schemes
&
\(\displaystyle
\begin{aligned}
A_{h_t}^{\mathrm{PE}}&=A_{h_t}^{\mathrm{SE}},
&
A_{h_t}^{\mathrm{APE}}&=A_{h_t}^{\mathrm{ASE}},
\\
A_{h_t}^{\mathrm{EVPE}}&=A_{h_t}^{\mathrm{EVSE}},
&
A_{h_t}^{\mathrm{EVAPE}}&=A_{h_t}^{\mathrm{EVASE}}.
\end{aligned}
\)
\\[0.3cm]
\hline
\renewcommand{\arraystretch}{0.5cm}
\end{tabular}
\captionof{table}{Jacobians of the linearized reaction maps associated with the schemes $\mathcal{S} \in \mathfrak{S}$. The quantities $s_g:=1-h_tg_v$ and $s_f:=1-h_tf_u$ are introduced for compactness.}
\label{tab:reaction_jacobians}
\end{table}

\noindent
In order to analyze fully-discrete
Turing instabilities of first-order symplectic schemes here considered, we will make use of the Jury stability criterion provided in Proposition \ref{prop:jury_turing}. Their
instability region will be characterized by the violation of at least one of these conditions for some $\mu>0$, while they remain satisfied at $\mu=0$.

Firstly, we establish  propositions that clarify the relationship
between the continuous and fully discrete stability conditions in terms of the
Jury characterization, thereby distinguishing instability mechanisms inherited
from the continuous problem from purely spurious numerical instabilities.   
\begin{proposition}
\label{prop:J3_monotonicity}
For all $\mathrm{IE}$-\(\mathcal S\) and $\mathrm{EX}$-\(\mathcal S\) families, if homogeneous
stability holds, then the third Jury condition cannot be violated.
\end{proposition}

\begin{proof}
By definition of the third Jury quantity, we have
$
\mathcal J_3(\mu)
=
1-b_u^{\mathcal D}(\mu)b_v^{\mathcal D}(\mu)\det\bigl(A_{h_t}^{\mathcal S}\bigr).$
If homogeneous stability holds, then in particular
$
\mathcal J_3(0)>0.$
Since \(b_u^{\mathcal D}(0)b_v^{\mathcal D}(0)=1\), this implies
$
1-\det\bigl(A_{h_t}^{\mathcal S}\bigr)>0.$
Moreover, for all \(\mu\ge0\),
$
0<b_u^{\mathcal D}(\mu)b_v^{\mathcal D}(\mu)\le 1.$
Hence, if \(\det(A_{h_t}^{\mathcal S})\ge0\), then
\[
\mathcal J_3(\mu)
=
1-b_u^{\mathcal D}(\mu)b_v^{\mathcal D}(\mu)\det\bigl(A_{h_t}^{\mathcal S}\bigr)
\ge
1-\det\bigl(A_{h_t}^{\mathcal S}\bigr)
=
\mathcal J_3(0)>0.
\]
On the other hand, if \(\det(A_{h_t}^{\mathcal S})<0\), then
\[
\mathcal J_3(\mu)
=
1-b_u^{\mathcal D}(\mu)b_v^{\mathcal D}(\mu)\det\bigl(A_{h_t}^{\mathcal S}\bigr)
>1>0.
\]
Therefore, \(\mathcal J_3(\mu)>0\) for all \(\mu\ge0\), and the third Jury
condition cannot be violated.
\end{proof}

%

\begin{proposition}[Continuous interpretation]
\label{prop:continuous_interpretation}
For every first-order splitting scheme considered above, the modal amplification
matrix admits the second-order expansion
\[
G_{h_t}(\mu) = I + h_t\,J(\mu) + h_t^{2}\,G_2(\mu) + O(h_t^{3}),
\qquad J(\mu) := J^{*} - \mu D,
\]
and the first two Jury quantities satisfy
\[
\mathcal J_1(\mu) = h_t^{2}\,h(\mu) + O(h_t^{3}),
\qquad
\mathcal J_2(\mu) = 4 + 2\,h_t\,\operatorname{tr}\!\big(J(\mu)\big) + O(h_t^{2}).
\]
Hence, the discrete stability condition $\mathcal J_1(\mu)>0$ corresponds,
asymptotically as $h_t\to 0$, to the continuous condition $h(\mu)>0$, while the
condition $\mathcal J_2(\mu)>0$ is a purely discrete stability requirement
without a continuous counterpart.
\end{proposition}

\begin{proof}
We expand separately the diffusion and the reaction substeps to second order in
$h_t$, multiply the two expansions, and read off the trace and determinant of
$G_{h_t}(\mu)$.

\medskip
\noindent\emph{Diffusion substep.}
For both the implicit Euler and the exact treatment of the diffusive flow the
modal factors $b_u^{\mathcal D},b_v^{\mathcal D}$ are scalar functions of $h_t\mu$ with the common
expansion
\[
b_u^{\mathcal D}(\mu) = 1 - h_t D_u\mu + \alpha\,h_t^{2}D_u^{2}\mu^{2} + O(h_t^{3}),
\qquad
b_v^{\mathcal D}(\mu) = 1 - h_t D_v\mu + \alpha\,h_t^{2}D_v^{2}\mu^{2} + O(h_t^{3}),
\]
where $\alpha = 1$ for the IE scheme and $\alpha = \tfrac12$ for the EX scheme.
Thus $B_{h_t}^{\mathcal D}(\mu) = \operatorname{diag}\!\big(b_u^{\mathcal D},b_v^{\mathcal D}\big)
= I - h_t\mu D + \alpha\,h_t^{2}\mu^{2}D^{2} + O(h_t^{3})$.

\medskip
\noindent\emph{Reaction substep.}
Each linearized reaction map of Table~\ref{tab:reaction_jacobians} is analytic in $h_t$ and reduces to the
identity at $h_t=0$, so it admits an expansion of the form
\begin{equation}
\label{eq:reaction_expansion}
A_{h_t}^{\mathcal S} = I + h_t J^{*} + h_t^{2} M^{\mathcal S} + O(h_t^{3}),
\qquad
M^{\mathcal S} = \begin{pmatrix} m_{11} & m_{12}\\ m_{21} & m_{22}\end{pmatrix},
\end{equation}
where the second-order matrix $M^{\mathcal S}$ depends on the scheme. For instance, for
the EE reaction $M^{\mathrm{EE}}=0$, while for the ASE reaction a direct
expansion of the entries in Table~~\ref{tab:reaction_jacobians} (using $s_g=1-h_tg_v$ and
$(1-h_tg_v)^{-1}=1+h_tg_v+O(h_t^2)$) gives
\[
A^{\mathrm{ASE}}_{h_t}=
\begin{pmatrix}
1+h_tf_u+h_t^{2}f_vg_u+O(h_t^3) & h_tf_v+h_t^{2}f_vg_v+O(h_t^3)\\[2pt]
h_tg_u+h_t^{2}g_ug_v+O(h_t^3) & 1+h_tg_v+h_t^{2}g_v^{2}+O(h_t^3)
\end{pmatrix},
\]
that is $M^{\mathrm{ASE}}=\left(\begin{smallmatrix} f_vg_u & f_vg_v\\ g_ug_v & g_v^{2}\end{smallmatrix}\right)$.
The explicit form of $M^{\mathcal S}$ will not be needed: only the fact that it enters at
order $h_t^{2}$ matters below.

\medskip
\noindent\emph{Product expansion.}
Multiplying the two expansions and collecting powers of $h_t$,
\[
G_{h_t}(\mu) = B_{h_t}^{\mathcal D}(\mu)\,A_{h_t}^{\mathcal S}
= I + h_t\big(J^{*}-\mu D\big) + h_t^{2}\,G_2(\mu) + O(h_t^{3}),
\]
which proves the first claim with $J(\mu)=J^{*}-\mu D$ and
$G_2(\mu) = M^{\mathcal S} - \mu D J^{*} + \alpha\mu^{2}D^{2}$.
Writing $G_{h_t}(\mu)=\big(G_{ij}\big)$, the entrywise expansions are
\[
\begin{aligned}
G_{11} &= 1 + h_t(f_u-\mu D_u) + h_t^{2}\big(m_{11}-\mu D_u f_u+\alpha\mu^{2}D_u^{2}\big)+O(h_t^3),\\
G_{22} &= 1 + h_t(g_v-\mu D_v) + h_t^{2}\big(m_{22}-\mu D_v g_v+\alpha\mu^{2}D_v^{2}\big)+O(h_t^3),\\
G_{12} &= h_t f_v + h_t^{2}\big(m_{12}-\mu D_u f_v\big)+O(h_t^3),\\
G_{21} &= h_t g_u + h_t^{2}\big(m_{21}-\mu D_v g_u\big)+O(h_t^3).
\end{aligned}
\]
Note, in particular, the cross term $-\mu D_v g_v$ in $G_{22}$ (and $-\mu D_u f_u$
in $G_{11}$), which originates from the product of the linear diffusion term
$-h_t\mu D_v$ with the diagonal reaction term $1+h_tg_v$.

\medskip
\noindent\emph{Trace and determinant.}
For the trace,
\[
\tau_{h_t}(\mu) = G_{11}+G_{22}
= 2 + h_t\,\operatorname{tr}\!\big(J(\mu)\big)
+ h_t^{2}\,S + O(h_t^{3}),
\]
where, writing $\operatorname{tr}(J(\mu))=(f_u+g_v)-\mu(D_u+D_v)$, the
second-order coefficient is
\[
S := (m_{11}+m_{22}) - \mu\big(D_u f_u + D_v g_v\big) + \alpha\mu^{2}\big(D_u^{2}+D_v^{2}\big).
\]
For the determinant, $\delta_{h_t}(\mu)=G_{11}G_{22}-G_{12}G_{21}$, a direct
computation of the $h_t^{2}$ coefficient gives the same quantity $S$ augmented by
the mixed contribution coming from the $O(h_t)$ entries,
\[
\delta_{h_t}(\mu)
= 1 + h_t\,\operatorname{tr}\!\big(J(\mu)\big)
+ h_t^{2}\big(S + h(\mu)\big) + O(h_t^{3}),
\]
where
\[
h(\mu)=D_uD_v\mu^{2}-(D_ug_v+D_vf_u)\mu+(f_ug_v-f_vg_u)
\]
is precisely the continuous Turing polynomial. Significantly, the extra term $h(\mu)$
does not involve the entries $m_{ij}$ of $M^{\mathcal S}$: it arises only from the
interaction between the diffusion factors and the \emph{linear} part $J^{*}$ of
the reaction map. Hence the computation below is the same for every scheme
$\mathcal S$.

\medskip
\noindent\emph{Jury quantities.}
Using $\mathcal J_1 = 1-\tau_{h_t}+\delta_{h_t}$, the constant terms cancel, the $O(h_t)$ terms cancel, and the second-order
terms cancel as well, leaving exactly the mixed contribution:
\[
\mathcal J_1(\mu)
= \big(1-2+1\big)
+ h_t\big(-\operatorname{tr}J(\mu)+\operatorname{tr}J(\mu)\big)
+ h_t^{2}\big(-S + S + h(\mu)\big) + O(h_t^{3})
= h_t^{2}\,h(\mu) + O(h_t^{3}).
\]
 For the second Jury quantity
$\mathcal J_2 = 1+\tau_{h_t}+\delta_{h_t}$ the constant terms add up to $4$ and
the $O(h_t)$ terms add up to $2\operatorname{tr}(J(\mu))$, while the second-order
terms now add:
\[
\mathcal J_2(\mu)
= 4 + 2h_t\,\operatorname{tr}\!\big(J(\mu)\big)
+ h_t^{2}\big(2S + h(\mu)\big) + O(h_t^{3}).
\]
Discarding the $O(h_t^{2})$ remainder yields
\[
\mathcal J_2(\mu) = 4 + 2h_t\,\operatorname{tr}\!\big(J(\mu)\big) + O(h_t^{2}),
\]
as claimed. In contrast with $\mathcal J_1$, the $h_t^{2}$ contributions to
$\mathcal J_2$ do not cancel, which already indicates that the second Jury
condition encodes information of a purely discrete nature, with no counterpart in
the continuous Turing polynomial.
\end{proof}

\begin{remark}
 Since the $h_t^{2}$ term of
$\mathcal J_1$ equals $h(\mu)$ \emph{independently} of the scheme-dependent
matrix $M^{S}$, all first-order families share the same leading-order first Jury
quantity. The schemes differ only at order $h_t^{3}$ in $\mathcal J_1$, and
already at order $h_t^{2}$ in $\mathcal J_2$ through the term $2S+h(\mu)$. The
exact (non-asymptotic) identity $P_{\sigma_1}^{\mathrm{IMEX}}(\mu)=h_t^{2}h(\mu)$
of Remark~\ref{prop:EE_J1_structure} is the special case in which $M^{\mathrm{EE}}=0$ makes all
higher-order corrections to $\mathcal J_1$ vanish.
\end{remark}

\begin{proposition}[Homogeneous-mode stability thresholds]
\label{prop:homogeneous_thresholds}
Assume that the kinetic equilibrium is linearly stable, i.e. that
\eqref{eq:conditions_lin} holds. 
For each scheme \(\mathcal S\), define
\[
\bar h_{\mathrm{hom}}^{\mathcal S}
:=
\sup\Bigl\{
h_t>0:
\mathcal J_1(0)>0,\;
\mathcal J_2(0)>0,\;
\mathcal J_3(0)>0
\Bigr\}.
\]
Then the kinetic equilibrium is stable for the homogeneous mode of the
numerical scheme \(\mathcal S\) only for time steps satisfying
\[
0<h_t<\bar h_{\mathrm{hom}}^{\mathcal S}.
\]
Using the same notation as given in Equations \ref{eq:conditions_lin}, the homogeneous stability thresholds are reported in Table \ref{tab:hom_thresholds}.
\begin{table}[ht]
\centering
\renewcommand{\arraystretch}{1}
\begin{tabular}{ll}
\hline
Family & Threshold \\
\hline
\noalign{\vskip 2pt}
EE &
$\displaystyle
\bar h_{\mathrm{hom}}^{\mathrm{EE}}
=
\begin{cases}
\dfrac{-\tau^*-\sqrt{(\tau^*)^2-4\delta^*}}{\delta^*},
& (\tau^*)^2\ge4\delta^*,\\[1ex]
\dfrac{-\tau^*}{\delta^*},
& (\tau^*)^2<4\delta^*
\end{cases}
$
\\[3ex]
\noalign{\vskip 2pt}
ASE/APE &
$\displaystyle
\bar h_{\mathrm{hom}}^{\mathrm{ASE}}
=
\min\!\left\{
\gamma_g,\,
\frac{f_u-g_v+\sqrt{(f_u-g_v)^2+4\delta^*}}{\delta^*}
\right\}
$
\\[2ex]

SE/PE &
$\displaystyle
\bar h_{\mathrm{hom}}^{\mathrm{SE}}
=
\min\!\left\{
\gamma_f,\,
\frac{g_v-f_u+\sqrt{(g_v-f_u)^2+4\delta^*}}{\delta^*}
\right\}
$
\\[2ex]

EVSE, EVPE, EVASE, EVAPE &
$\displaystyle
\bar h_{\mathrm{hom}}^{\mathrm{EV}}
=
\begin{cases}
\dfrac{-\tau^*-\sqrt{(\tau^*)^2-4\beta}}{\beta},
& \beta<0,\\[1ex]
-\dfrac{2}{\tau^*},
& \beta=0
\end{cases}
$
\\[6pt]
\hline
\end{tabular}

\medskip

\[
\gamma_g=
\begin{cases}
1/g_v,& g_v>0,\\
+\infty,& g_v\le0,
\end{cases}
\qquad
\gamma_f=
\begin{cases}
1/f_u,& f_u>0,\\
+\infty,& f_u\le0,
\end{cases}
\qquad
\beta=f_u g_v+f_v g_u.
\]
\caption{Homogeneous stability thresholds.}
\label{tab:hom_thresholds}
\end{table}








\end{proposition}

\begin{proof}
We consider the homogeneous mode \(\mu=0\). In this case, for both implicit
and explicit diffusion treatments, one has
\(
b_u^{\mathcal D}=b_v^{\mathcal D}=1.
\)
Therefore, the amplification matrix of the numerical scheme satisfies
\(
G_{h_t}(0)=A_{h_t}^{\mathcal S}.
\)
In particular, the trace and determinant \(\tau_{h_t}(0) \) and
\(\delta_{h_t}(0)\) of \(G_{h_t}(0)\) coincide with
those of the corresponding matrix \(A_{h_t}^{\mathcal S}\). Hence the
homogeneous-mode stability of the numerical scheme is determined by applying
the Jury conditions \eqref{eq:Jury_cond}-\eqref{eq:Jury} to the matrices
\(A_{h_t}^{\mathcal S}\) listed in Table~\ref{tab:reaction_jacobians}. Solving these inequalities with respect to \(h_t\) gives the thresholds
stated above.

For the EV families, the corresponding threshold depends on
\(
\beta:=f_u g_v+f_v g_u
\). From 
\(
f_u g_v=\frac{\beta+\delta^*}{2},
\)
the assumption \(\delta^*>0\) implies
\(
f_u g_v>0,
\)
which cannot occur for $\beta >0$ under the hypothesis  \(f_u g_v<0\) in Equation (\ref{eq:cond_fugv_neg}). Therefore, whenever the condition
\(f_u g_v<0\) is imposed, only the case \(\beta\leq0\) can occur.

\end{proof}
\subsection{Specialization of Jury conditions to first-order symplectic families}
\label{sec:jury_families}

In the following, we assume that each numerical scheme is applied with a
time step \(h_t\) such that the stability of the homogeneous equilibrium is
preserved. More precisely, in view of Proposition~\ref{prop:homogeneous_thresholds},
we assume that
\(
0<h_t<\bar h_{\mathrm{hom}}^{\mathcal S}.
\)
Under this assumption, the homogeneous mode \(\mu=0\) is stable for the
scheme \(\mathcal S\).

Moreover, as shown in Proposition~\ref{prop:J3_monotonicity},  if the third Jury condition is satisfied at the homogeneous
mode, then it remains satisfied for all admissible modes considered below.
Therefore, in the rest of this section, we specialize only the first two
Jury conditions to the first-order symplectic splitting families.

\begin{proposition}
\label{prop:negative_jury_some_mu}
Let \(\mathcal{S}\) be one of the reaction maps listed in Table~\ref{tab:reaction_jacobians}, and let
\(
A^\mathcal{S}_{h_t}=
\begin{pmatrix}
a_{11}^\mathcal{S} & a_{12}^\mathcal{S}\\
a_{21}^\mathcal{S} & a_{22}^\mathcal{S}
\end{pmatrix}
\)
be the Jacobian of the corresponding linearized reaction substep at the homogeneous equilibrium. 

Define the  functions
\begin{equation}\label{eq:J12compact}
    J^{\mathcal D-\mathcal S}_{i}(\mu)
=
\bigl(1+\sigma_i b_u^{\mathcal D}(\mu)a_{11}^\mathcal{S}\bigr)
\bigl(1+\sigma_i b_v^{\mathcal D}(\mu)a_{22}^\mathcal{S}\bigr)
-
b_u^{\mathcal D}(\mu)b_v^{\mathcal D}(\mu)a_{12}^\mathcal{S}a_{21}^\mathcal{S}, \quad \sigma_i=2i-3, \quad i=1,2. 
\end{equation}
Then, for the family of symplectic schemes  \(\mathcal D- \mathcal S\), with $\mathcal D \in (\mathrm{IE},\mathrm{EX})$ and $\mathcal{S} \in \mathfrak{S}$, \(\sigma_1=-1\) corresponds to the first Jury function \(J_1\), while
\(\sigma_2=+1\) corresponds to the second Jury function \(J_2\).

Moreover, for the  family \(\mathrm{IE}-\mathcal S\), 
\[
J^{\mathrm{IE}-\mathcal S}_{i}(\mu)
=
\frac{P_{\sigma_i}^{\mathcal S}(\mu)}
{(1+\alpha\mu)(1+\beta\mu)}, \quad 
\alpha=h_tD_u,\qquad \beta=h_tD_v.
\]
where
\(
P_{\sigma_i}^{\mathcal S}(\mu)
=
A\mu^2+B_{\sigma_i}^\mathcal S\mu+C_{\sigma_i}^{\mathcal S},\) with \[A=\alpha\beta>0, \quad  B_{\sigma_i}^\mathcal S
=
\beta(1+\sigma_i a_{11}^{\mathcal S})
+
\alpha(1+\sigma_i a_{22}^{\mathcal S}), \quad C_{\sigma_i}^{\mathcal S}
=
(1+\sigma_i a_{11}^{\mathcal S})(1+\sigma_i a_{22}^{\mathcal S})
-
a_{12}^{\mathcal S}a_{21}^{\mathcal S}.
\]

For the family \(\mathrm{EX}-{\mathcal S}\), one has
\[
J^{\mathrm{EX}-{\mathcal S}}_{i}(\mu)
=
1
+
\sigma a_{11}^{\mathcal S}e^{-h_tD_u\mu}
+
\sigma a_{22}^{\mathcal S}e^{-h_tD_v\mu}
+
\det(A^{\mathcal S}_{h_t})e^{-h_t(D_u+D_v)\mu}.
\]
Equivalently, setting
\(
x=e^{-h_tD_u\mu},\) and 
\(r=\frac{D_v}{D_u},\)
the change of variable \(\mu>0\) maps onto \(x\in(0,1)\), and one obtains
\[
J^{\mathrm{EX}-{\mathcal S}}_{i}(\mu)
=
F_i^{\mathcal S}(x),
\qquad x\in(0,1),
\]
where
\begin{equation}\label{eq:exponential}
F_i^{\mathcal S}(x)
=
1+\sigma_i a_{11}^{\mathcal S} x
+\sigma_i a_{22}^{\mathcal S} x^r
+\det(A_{h_t}^{\mathcal S})x^{1+r}.
\end{equation}
If \(r=D_v/D_u\) is an integer, \(F_{i}^{\mathcal S}\) is a polynomial in \(x\) on the interval \((0,1)\); otherwise it is a generalized polynomial.
\end{proposition}

\begin{proof}
The trace and determinant of
\(
G^{\mathcal D-{\mathcal S}}_{h_t}(\mu)
=
B^{\mathcal D}_{h_t}(\mu)A^{\mathcal S}_{h_t}
\)
are
\[
\tau^{\mathcal D-{\mathcal S}}_{h_t}(\mu)
=
b_u^{\mathcal D}(\mu)a_{11}^{\mathcal S}
+
b_v^{\mathcal D}(\mu)a_{22}^{\mathcal S},\qquad \delta^{\mathcal D-{\mathcal S}}_{h_t}(\mu)
=
b_u^{\mathcal D}(\mu)b_v^{\mathcal D}(\mu)
\det(A^{\mathcal S}_{h_t}).
\]
Since
\(
\det(A^{\mathcal S}_{h_t})
=
a_{11}^{\mathcal S}a_{22}^{\mathcal S}-a_{12}^{\mathcal S}a_{21}^{\mathcal S},
\)
the Jury functions
\[
J^{\mathcal D-{\mathcal S}}_{1}=1-\tau^{\mathcal D-{\mathcal S}}_{h_t}+\delta^{\mathcal D-S}_{h_t},
\qquad
J^{\mathcal D-{\mathcal S}}_{2}=1+\tau^{\mathcal D-{\mathcal S}}_{h_t}+\delta^{\mathcal D-{\mathcal S}}_{h_t}
\]
can be written in the unified form stated in \eqref{eq:J12compact}.

For the \(\mathrm{IE}-{\mathcal S}\) family, substituting
\(
b_u^{\mathrm{IE}}(\mu)=\frac{1}{1+\alpha\mu},
\,
b_v^{\mathrm{IE}}(\mu)=\frac{1}{1+\beta\mu},
\)
and multiplying by the positive denominator
\(
(1+\alpha\mu)(1+\beta\mu)>0
\)
gives the quadratic polynomial \(P_{\sigma}^{\mathcal S}(\mu)\). 

For the \(\mathrm{EX}-{\mathcal S}\) family, the substitution
\(
b_u^{\mathrm{EX}}(\mu)=e^{-h_tD_u\mu},
\,
b_v^{\mathrm{EX}}(\mu)=e^{-h_tD_v\mu}
\)
gives the exponential expression in \eqref{eq:exponential}. 
\end{proof}

\subsection{Turing instability analysis for the IE--\(\mathcal{S}\) families}
\label{subsec:IE_S_instability}

In view of Proposition~\ref{prop:negative_jury_some_mu}, we now restrict the
algebraic instability analysis to the \(\mathrm{IE}-\mathcal{S}\) families.
Indeed, in this case the denominator appearing in the Jury functions is strictly
positive for every \(\mu\geq0\), so that the sign of each Jury function is
completely determined by the sign of the corresponding quadratic polynomial
\(P_{\sigma_i}^{\mathcal{S}}\). This allows us to obtain conditions for the occurrence of
a non-homogeneous instability which do not explicitly involve the modal variable
\(\mu\).

More precisely, 
\(
P_{\sigma_i}^{\mathcal S}(\mu)
=
A\mu^2+B_{\sigma_i}^\mathcal S\mu+C_{\sigma_i}^{\mathcal S},\) 
with $A = D_u D_v h_t^2>0$, and denote by
\[
\Delta_{\sigma_i}^{\mathcal{S}}
=
\left(B_{\sigma_i}^{\mathcal{S}}\right)^2
-
4AC_{\sigma_i}^{\mathcal{S}}
\]
its discriminant. Since \(P_{\sigma_i}^{\mathcal{S}}\) is a convex quadratic polynomial,
the \(i\)-th Jury condition is violated for some positive non-homogeneous mode
if and only if
\[
C_{\sigma_i}^{\mathcal{S}}<0
\]
or
\[
B_{\sigma_i}^{\mathcal{S}}<0
\qquad\text{and}\qquad
\Delta_{\sigma_i}^{\mathcal{S}}>0.
\]
Under the standing assumption that the homogeneous mode is stable, one has
\[
C_{\sigma_i}^{\mathcal{S}}
=
J_{\sigma_i}^{\mathrm{IE}-\mathcal{S}}(0)>0.
\]
Hence, in this case, the criterion reduces to
\[
B_{\sigma_i}^{\mathcal{S}}<0
\qquad\text{and}\qquad
\Delta_{\sigma_i}^{\mathcal{S}}>0.
\]

This reduction is not available in the same form for the
\(\mathrm{EX}-\mathcal{S}\) families. After the change of variable introduced
above, the corresponding Jury functions become generalized polynomials in \(x\),
and, when \(D_v/D_u\) is an integer, polynomials of degree \(1+D_v/D_u\). In
regimes with large diffusion contrast, this degree may be very high, so that no
comparably simple mode-independent sign criterion is obtained. For this reason,
the following explicit conditions are derived only for the
\(\mathrm{IE}-\mathcal{S}\) families, while the \(\mathrm{EX}-\mathcal{S}\)
schemes will be assessed through the modal diagnostics in the numerical section.

\subsubsection{Explicit Euler reaction (EE)}

For the explicit Euler reaction step, \(A_{h_t}^{\mathrm{EE}}=I+h_tJ^*\), namely
\(a_{11}^{\mathrm{EE}}=1+h_tf_u\), \(a_{12}^{\mathrm{EE}}=h_tf_v\),
\(a_{21}^{\mathrm{EE}}=h_tg_u\), and \(a_{22}^{\mathrm{EE}}=1+h_tg_v\). Therefore, for the
IMEX scheme, i.e. the \(\mathrm{IE}\)-\(\mathrm{EE}\) scheme, the polynomial associated
with the \(i\)-th Jury quantity is
\[
P_{\sigma_i}^{\mathrm{IMEX}}(\mu)
=
A\mu^2+B_{\sigma_i}^{\mathrm{EE}}\mu+C_{\sigma_i}^{\mathrm{EE}},
\qquad i=1,2.
\] By using the notation in \eqref{eq:conditions_lin} and
introducing \(m:=D_ug_v+D_vf_u\), for \(\sigma_1=-1\), corresponding to the
first Jury quantity, one has
\[
B_{\sigma_1}^{\mathrm{EE}}=-h_t^2\, m,
\qquad
C_{\sigma_1}^{\mathrm{EE}}=h_t^2\, \delta^*,
\]
and hence
\[
\Delta_{\sigma_1}^{\mathrm{EE}}
=
h_t^4
\left[
m^2
-
4D_uD_v\, \delta^*
\right].
\]
For \(\sigma_2=+1\), corresponding to the second Jury quantity, one obtains
\[
B_{\sigma_2}^{\mathrm{EE}}
=
h_t\left[h_t\, m+2(D_u+D_v)\right],
\qquad 
C_{\sigma_2}^{\mathrm{EE}}
=
h_t^2\, \delta^*+2h_t\, \tau^*+4,
\]
with discriminant
\[
\Delta_{\sigma_2}^{\mathrm{EE}}
=
h_t^2 \left\{
h_t^2 \left[ m^2 - 4D_uD_v\delta^* \right]
+
4h_t \left[ (D_u+D_v)\,m - 2D_uD_v\, \tau^* \right]
+
4(D_u-D_v)^2
\right\}.
\]

\subsubsection{Adjoint symplectic or adjoint Poisson Euler reaction (ASE/APE)}

Setting \(s_g:=1-h_tg_v\), for the \(\mathrm{IE}\)-ASE/APE scheme the polynomial
associated with the \(i\)-th Jury quantity is
\[
P_{\sigma_i}^{\mathrm{ASE}}(\mu)
=
A\mu^2+B_{\sigma_i}^{\mathrm{ASE}}\mu+C_{\sigma_i}^{\mathrm{ASE}},
\qquad i=1,2.
\]
where \(A=h_t^2D_uD_v\). For \(\sigma_1=-1\), corresponding to the first Jury quantity, one has
\[
B_{\sigma_1}^{\mathrm{ASE}}
=
\frac{h_t^2}{s_g}\left(h_tD_v\delta^*-m\right),
\qquad
C_{\sigma_1}^{\mathrm{ASE}}
=
\frac{h_t^2}{s_g}\delta^*,
\]
and hence
\[
\Delta_{\sigma_1}^{\mathrm{ASE}}
=
\frac{h_t^4}{s_g^2}
\left[
m^2-4D_uD_v\delta^*
+
2h_tD_v\delta^*(D_ug_v-D_vf_u)
+
h_t^2D_v^2(\delta^*)^2
\right].
\]

For \(\sigma_2=+1\), corresponding to the second Jury quantity, and introducing
\(\ell:=D_v(f_u-2g_v)-D_ug_v\), one obtains
\[
B_{\sigma_2}^{\mathrm{ASE}}
=
\frac{h_t}{s_g}
\left[
-h_t^2D_v\delta^*+h_t\ell+2(D_u+D_v)
\right],
\qquad 
C_{\sigma_2}^{\mathrm{ASE}}
=
\frac{1}{s_g}
\left[
-h_t^2\delta^*+2h_t(f_u-g_v)+4
\right],
\]
with discriminant
\[
\Delta_{\sigma_2}^{\mathrm{ASE}}
=
\frac{h_t^2}{s_g^2}
\left\{
\left[
-h_t^2D_v\delta^*+h_t\ell+2(D_u+D_v)
\right]^2
-
4D_uD_v s_g
\left[
-h_t^2\delta^*+2h_t(f_u-g_v)+4
\right]
\right\}.
\]

\subsubsection{Symplectic or Poisson Euler reaction (SE/PE)}

Setting \(s_f:=1-h_tf_u\), for the \(\mathrm{IE}\)-SE/PE scheme the polynomial
associated with the \(i\)-th Jury quantity is
\[
P_{\sigma_i}^{\mathrm{SE}}(\mu)
=
A\mu^2+B_{\sigma_i}^{\mathrm{SE}}\mu+C_{\sigma_i}^{\mathrm{SE}},
\qquad i=1,2.
\]
For \(\sigma_1=-1\), corresponding to the first Jury
quantity, one has
\[
B_{\sigma_1}^{\mathrm{SE}}
=
\frac{h_t^2}{s_f}\left(h_tD_u\delta^*-m\right),
\qquad
C_{\sigma_1}^{\mathrm{SE}}
=
\frac{h_t^2}{s_f}\delta^*,
\]
and hence
\[
\Delta_{\sigma_1}^{\mathrm{SE}}
=
\frac{h_t^4}{s_f^2}
\left[
m^2-4D_uD_v\delta^*
+
2h_tD_u\delta^*(D_vf_u-D_ug_v)
+
h_t^2D_u^2(\delta^*)^2
\right].
\]

For \(\sigma_2=+1\), corresponding to the second Jury quantity, and introducing
\(
\ell_f:=D_ug_v-(2D_u+D_v)f_u,
\)
one obtains
\[
B_{\sigma_2}^{\mathrm{SE}}
=
\frac{h_t}{s_f}
\left[
-h_t^2D_u\delta^*+h_t\ell_f+2(D_u+D_v)
\right],
\qquad
C_{\sigma_2}^{\mathrm{SE}}
=
\frac{1}{s_f}
\left[
-h_t^2\delta^*+2h_t(g_v-f_u)+4
\right],
\]
with discriminant
\[
\Delta_{\sigma_2}^{\mathrm{SE}}
=
\frac{h_t^2}{s_f^2}
\left\{
\left[
-h_t^2D_u\delta^*+h_t\ell_f+2(D_u+D_v)
\right]^2
-
4D_uD_v s_f
\left[
-h_t^2\delta^*+2h_t(g_v-f_u)+4
\right]
\right\}.
\]

\subsubsection{Explicit variants of symplectic and Poisson Euler reaction (EVSE/EVPE)}

For the \(\mathrm{IE}\)-EVSE/EVPE scheme the polynomial associated with the
\(i\)-th Jury quantity is
\[
P_{\sigma_i}^{\mathrm{EVSE}}(\mu)
=
A\mu^2+B_{\sigma_i}^{\mathrm{EVSE}}\mu+C_{\sigma_i}^{\mathrm{EVSE}},
\qquad i=1,2.
\]
For \(\sigma_1=-1\), corresponding to the first Jury
function, one has
\[
B_{\sigma_1}^{\mathrm{EVSE}}
=
-h_t^2\left(m+h_tD_uf_vg_u\right),
\qquad
C_{\sigma_1}^{\mathrm{EVSE}}
=
h_t^2\delta^*,
\]
and hence
\[
\Delta_{\sigma_1}^{\mathrm{EVSE}}
=
h_t^4
\left[
m^2-4D_uD_v\delta^*
+
2h_tD_umf_vg_u
+
h_t^2D_u^2(f_vg_u)^2
\right].
\]

For \(\sigma_2=+1\), corresponding to the second Jury quantity, one obtains
\[
B_{\sigma_2}^{\mathrm{EVSE}}
=
h_t
\left[
h_t^2D_uf_vg_u+h_tm+2(D_u+D_v)
\right],
\qquad
C_{\sigma_2}^{\mathrm{EVSE}}
=
h_t^2(f_ug_v+f_vg_u)+2h_t \tau^*+4,
\]
with discriminant
\[
\Delta_{\sigma_2}^{\mathrm{EVSE}}
=
h_t^2
\left\{
4(D_u-D_v)^2
+
4h_t\left[(D_u+D_v)m-2D_uD_v\tau^*\right]
\right.
\]
\[
\left.
+
h_t^2
\left[
m^2+4D_u(D_u+D_v)f_vg_u
-
4D_uD_v(f_ug_v+f_vg_u)
\right]
+
2h_t^3D_umf_vg_u
+
h_t^4D_u^2(f_vg_u)^2
\right\}.
\]

\subsubsection{Explicit variants of adjoint symplectic and Poisson Euler reaction (EVASE/EVAPE)}

For the \(\mathrm{IE}\)-EVASE/EVAPE scheme the polynomial associated with the
\(i\)-th Jury quantity is
\[
P_{\sigma_i}^{\mathrm{EVASE}}(\mu)
=
A\mu^2+B_{\sigma_i}^{\mathrm{EVASE}}\mu+C_{\sigma_i}^{\mathrm{EVASE}},
\qquad i=1,2.
\]
By using the notation in \eqref{eq:conditions_lin}-\eqref{eq:algebraic_continum}, for \(\sigma_1=-1\), corresponding to the first Jury quantity, one has
\[
B_{\sigma_1}^{\mathrm{EVASE}}
=
-h_t^2\left(m+h_tD_vf_vg_u\right),
\qquad
C_{\sigma_1}^{\mathrm{EVASE}}
=
h_t^2\delta^*,
\]
and hence
\[
\Delta_{\sigma_1}^{\mathrm{EVASE}}
=
h_t^4
\left[
m^2-4D_uD_v\delta^*
+
2h_tD_vmf_vg_u
+
h_t^2D_v^2(f_vg_u)^2
\right].
\]

For \(\sigma_2=+1\), corresponding to the second Jury function, one obtains
\[
B_{\sigma_2}^{\mathrm{EVASE}}
=
h_t
\left[
h_t^2D_vf_vg_u+h_tm+2(D_u+D_v)
\right],
\qquad
C_{\sigma_2}^{\mathrm{EVASE}}
=
h_t^2(f_ug_v+f_vg_u)+2h_t\tau^*+4,
\]
with discriminant
\[
\Delta_{\sigma_2}^{\mathrm{EVASE}}
=
h_t^2
\left\{
4(D_u-D_v)^2
+
4h_t\left[(D_u+D_v)m-2D_uD_v\tau^*\right]
\right.
\]
\[
\left.
+
h_t^2
\left[
m^2+4D_v(D_u+D_v)f_vg_u
-
4D_uD_v(f_ug_v+f_vg_u)
\right]
+
2h_t^3D_vmf_vg_u
+
h_t^4D_v^2(f_vg_u)^2
\right\}.
\]

We recall that the first Jury quantity is closely related to the Turing
 instability mechanism of the continuous problem. In particular, its loss
 of positivity corresponds to the discrete counterpart of the diffusion-driven
instability observed at the continuous level. On the other hand, the second
 Jury condition is more directly connected to the specific structure of the
 time-discretization method. Its violation may therefore generate
scheme-dependent instabilities, which can occur even in parameter regimes
 where the continuous problem does not undergo a Turing instability.
 
The following result shows that, among the schemes
considered here, the IMEX scheme preserves exactly the sign structure of
the Turing polynomial \(h(\mu)\).
\begin{remark}[Exact preservation of the Turing polynomial sign by IMEX]
\label{prop:EE_J1_structure}
For the $\mathrm{IMEX}$ scheme, namely the $\mathrm{IE}-\mathrm{EE}$ scheme, the numerator polynomial
associated with the first Jury quantity satisfies, exactly and not only
asymptotically as \(h_t\to0\),
\[
P_{\sigma_1}^{\mathrm{IMEX}}(\mu) = h_t^2 \, h(\mu),
\]
where \(h(\mu)\) is given in \eqref{eq:h_cont}. Since, by
Proposition~\ref{prop:negative_jury_some_mu},
\[
\mathcal J_1^{\mathrm{IMEX}}(\mu)
=
\frac{P_{\sigma_1}^{\mathrm{IMEX}}(\mu)}{(1+h_tD_u\mu)(1+h_tD_v\mu)}
=
\frac{h_t^2}{(1+h_tD_u\mu)(1+h_tD_v\mu)}\,h(\mu),
\]
and the denominator \((1+h_tD_u\mu)(1+h_tD_v\mu)>0\) for every \(\mu\ge0\), it
follows that
\[
\operatorname{sign}\bigl(\mathcal J_1^{\mathrm{IMEX}}(\mu)\bigr)
=
\operatorname{sign}\bigl(h(\mu)\bigr),
\qquad \mu\ge 0.
\]
\end{remark}
This shows that, among the first-order splitting families considered here,
the IMEX scheme is the only one that preserves exactly, and not only
asymptotically as \(h_t\to0\), the sign of the polynomial \(h(\mu)\)
associated with the onset of Turing instability. Hence, with respect to the
first Jury condition, the IMEX discretization reproduces the same
diffusion-driven instability mechanism as the continuous reaction--diffusion
problem.

Nevertheless, this exact preservation does not rule out purely numerical
instabilities. Indeed, even the IMEX scheme may lose stability through the
violation of the second Jury condition.

\subsection{A sufficient dominance criterion for Turing-region preservation}
\label{subsec:dominance}

The first Jury quantity carries the continuous diffusion-driven mechanism: as
\(h_t\to0\), and in some cases exactly, its sign is governed by the continuous
Turing polynomial \(h(\mu)\). The second Jury quantity, by contrast, has no
continuous counterpart, since \(\mathcal J_2(\mu)\to4\) as \(h_t\to0\). A fully
discrete method can therefore reproduce the continuous Turing region only if the
second Jury condition does not introduce an additional, purely numerical
instability on non-homogeneous modes.

We make this observation precise for IMEX, which is the only scheme considered here whose
first Jury quantity reproduces the sign of \(h(\mu)\) exactly
(Remark~\ref{prop:EE_J1_structure}). Let \(\mathcal M_h^+\) denote the retained
positive modal set of the spatial discretization. Away from the neutral
boundary \(h(\mu)=0\), Turing-region preservation means that, under
homogeneous-mode stability, a mode \(\mu\in\mathcal M_h^+\) is unstable for the
fully discrete map exactly when \(h(\mu)<0\).

\begin{proposition}[Exact preservation for IMEX]
\label{prop:imex_exact_preservation}
Let the kinetic equilibrium be stable, \eqref{eq:conditions_lin}, and let
\(0<h_t<\bar h_{\mathrm{hom}}^{\mathrm{EE}}\). Then the IMEX scheme
\(\mathrm{IE\!-\!EE}\) is Turing-region preserving on \(\mathcal M_h^+\) if and
only if
\begin{equation}
\label{eq:imex_exact_condition}
\mathcal J_1(\mu)>0 \ \Longrightarrow\ \mathcal J_2(\mu)>0
\qquad\text{for all } \mu\in\mathcal M_h^+ .
\end{equation}
Equivalently, \(\mathcal J_2\) may become nonpositive only on modes of the
continuous Turing band, where \(\mathcal J_1<0\) already accounts for the
instability.
\end{proposition}

\begin{proof}
Since \(h_t<\bar h_{\mathrm{hom}}^{\mathrm{EE}}\), the homogeneous mode is
stable and, by Proposition~\ref{prop:J3_monotonicity}, \(\mathcal J_3(\mu)>0\)
for all \(\mu\ge0\). A mode \(\mu\in\mathcal M_h^+\) is therefore unstable if and
only if \(\mathcal J_1(\mu)<0\) or \(\mathcal J_2(\mu)<0\). For IMEX,
\(\mathcal J_1(\mu)<0\) is equivalent to \(h(\mu)<0\)
(Remark~\ref{prop:EE_J1_structure}). Hence, the discrete and continuous
instabilities agree unless some mode has \(\mathcal J_1(\mu)>0\), that is
\(h(\mu)>0\), and is destabilized by \(\mathcal J_2(\mu)<0\); ruling out this
case is exactly \eqref{eq:imex_exact_condition}.
\end{proof}

A convenient
sufficient strengthening is the dominance
\(\mathcal J_2(\mu)\ge\mathcal J_1(\mu)\), which immediately gives
\(\mathcal J_2(\mu)>0\) whenever \(\mathcal J_1(\mu)>0\). Although stronger than
necessary, this condition reduces to explicit inequalities in the parameters and
in the time step.

The two Jury polynomials share the same leading coefficient, so their difference
is linear in the modal variable, and the dominance
\(\mathcal J_2(\mu)\ge\mathcal J_1(\mu)\) for all \(\mu\ge0\) is equivalent to
the pair of inequalities
\begin{equation}
\label{eq:dominance_strong}
a^{\mathcal S}_{11}+a^{\mathcal S}_{22}\ge 0,
\qquad
D_v a^{\mathcal S}_{11}+D_u a^{\mathcal S}_{22}\ge 0 ,
\end{equation}
obtained at \(\mu=0\) and at the first order in \(\mu\). At \(\mu=0\), both
implicit Euler and exact diffusion give \(b^{\mathcal D}_u(0)=b^{\mathcal D}_v(0)=1\), so
\(G_{h_t}(0)=A^{\mathcal S}_{h_t}\), and these two quantities admit the
homogeneous interpretation
\begin{equation}
\label{eq:dominance_I_jury}
a^{\mathcal S}_{11}+a^{\mathcal S}_{22}
=
\mathcal J_2(0)+\mathcal J_3(0)-2,
\end{equation}
\begin{equation}
\label{eq:dominance_II_split}
D_v a^{\mathcal S}_{11}+D_u a^{\mathcal S}_{22}
=
\frac{D_u+D_v}{2}\bigl(\mathcal J_2(0)+\mathcal J_3(0)-2\bigr)
+
\frac{D_v-D_u}{2}\bigl(a^{\mathcal S}_{11}-a^{\mathcal S}_{22}\bigr).
\end{equation}
Thus homogeneous-mode stability alone does not control the dominance condition:
the second inequality also depends on the diffusion-weighted diagonal balance of
the reaction amplification matrix.

For IMEX, with \(a^{\mathrm{EE}}_{11}=1+h_tf_u\),
\(a^{\mathrm{EE}}_{22}=1+h_tg_v\), \(\tau^*=f_u+g_v\) and \(m=D_ug_v+D_vf_u\),
the dominance inequalities \eqref{eq:dominance_strong} read
\begin{equation}
\label{eq:imex_dominance}
2+h_t\tau^*\ge0,
\qquad
(D_u+D_v)+h_t\,m\ge0 .
\end{equation}
By the strengthening above, \eqref{eq:imex_dominance} is then sufficient for the
sharp requirement \eqref{eq:imex_exact_condition}.

The role of the dominance, however, is confined to the Turing-stable regime.
Only in a Turing-stable regime
\(\mathcal J_2\) may turn negative
outside the band and the dominance becomes effective. There both inequalities of
\eqref{eq:imex_dominance} are active, and we set
\begin{equation}
\label{eq:imex_dominance_threshold}
h_t^{\mathrm{dom}}
:=
\begin{cases}
\min\!\left\{\dfrac{2}{|\tau^*|},\,\dfrac{D_u+D_v}{|m|}\right\},
\qquad m<0 \\
\dfrac{2}{|\tau^*|},
\qquad \qquad \qquad \qquad \quad0 < m < 2 \sqrt{D_u D_v \delta^*}.
\end{cases}
\end{equation}

\begin{tr}[Sufficient condition for Turing-region preservation of IMEX]
\label{thm:imex_sufficient}
Assume kinetic stability \eqref{eq:conditions_lin}, and let
\(\bar h_{\mathrm{hom}}^{\mathrm{EE}}\) be the homogeneous threshold of
Table~\ref{tab:hom_thresholds}. Then IMEX is Turing-region preserving on
\(\mathcal M_h^+\) provided that
\begin{itemize}
\item in a continuous Turing regime, where $m > 2 \sqrt{D_u D_v \delta^*}$,
\(\;0<h_t<\bar h_{\mathrm{hom}}^{\mathrm{EE}}\); or
\item in a Turing-stable regime, where $m < 2 \sqrt{D_u D_v \delta^*}$,
\(\;0<h_t<\min\bigl\{\,h_t^{\mathrm{dom}},\,\bar h_{\mathrm{hom}}^{\mathrm{EE}}\,\bigr\}\),
with \(h_t^{\mathrm{dom}}\) given by \eqref{eq:imex_dominance_threshold}.
\end{itemize}
\end{tr}

\begin{proof}
The bound \(h_t<\bar h_{\mathrm{hom}}^{\mathrm{EE}}\) secures homogeneous-mode
stability in both cases. 
If \(m> 2\sqrt{D_u D_v \delta^*}\), 
the
homogeneous bound is the only active restriction.
If \(m<2\sqrt{D_u D_v \delta^*}\), the
bound \(h_t<h_t^{\mathrm{dom}}\) makes both inequalities of
\eqref{eq:imex_dominance} hold, hence the dominance
\(\mathcal J_2\ge\mathcal J_1\) and, in turn,
\eqref{eq:imex_exact_condition}; the conclusion follows from
Proposition~\ref{prop:imex_exact_preservation}.
\end{proof}

The two motivating examples realize the two regimes of
Theorem~\ref{thm:imex_sufficient}. In Example~2, a real Turing regime, the dominance plays no role and the only active bound is the
homogeneous threshold \(\bar h_{\mathrm{hom}}^{\mathrm{EE}}\); IMEX detects the
continuous band. In Example~1, a Turing-stable regime with \(m<0\), the
dominance becomes effective: one finds \(h_t^{\mathrm{dom}}\approx0.30\), well
below the step \(h_t^{(1)}=0.67\); at that step the second inequality of
\eqref{eq:imex_dominance} fails, \(\mathcal J_2\) turns negative on a mode with
\(\mathcal J_1>0\), and the sharp condition \eqref{eq:imex_exact_condition} is
violated, producing the spurious \(\mathcal J_2\)-driven pattern.

The whole argument rests on the exact sign preservation of \(\mathcal J_1\), and
is therefore specific to IMEX. For the other first-order families
\(\mathcal J_1(\mu)=h_t^2h(\mu)+O(h_t^3)\)
(Proposition~\ref{prop:continuous_interpretation}), so the sign agreement with
\(h(\mu)\) is only asymptotic; even the sharp condition
\eqref{eq:imex_exact_condition} would then guarantee instability on the modes
where \(\mathcal J_1<0\), not on the continuous band \(h<0\). Controlling
\(\mathcal J_2\) cannot repair this, because the defect already lies in
\(\mathcal J_1\). In particular, in a continuous Turing regime. the adjoint
family may keep \(\mathcal J_1\) positive on modes of the Turing band, thereby
suppressing the instability that the discrete map should reproduce; no condition
on \(\mathcal J_2\) can restore that detection.

\section{Numerical investigation in the discretization parameters}
\label{sec:numerical_investigation}

The previous sections have shown that the continuous Turing polynomial and the
fully discrete Jury conditions answer two different questions. The continuous
criterion identifies the diffusion-driven instability of the differential
problem, whereas the fully discrete map may either create additional
non-homogeneous instabilities or suppress modes that are unstable in the
continuous model. The purpose of this section is to make this discrepancy
quantitative in the discretization parameters, and to measure how far each
scheme departs from the continuous Turing prediction as the time step varies.

The analysis developed in Section~\ref{sec:discrete_turing} also shows that the two diffusion
treatments considered in this paper require different numerical diagnostics.
For the IE--$\mathcal S$ family, the Jury conditions reduce to algebraic inequalities
involving the reaction parameters, the diffusion coefficients, and the time
step $h_t$. In this case $h_t$ can be treated as an additional parameter, and
the occurrence of a fully discrete non-homogeneous instability can be detected
without computing the discrete modal values $\mu_{ij}$.

By contrast, for the EX--$\mathcal S$ family, the Jury functions contain exponential
terms depending explicitly on $\mu$. Therefore, the stability diagnosis
depends on the retained spatial spectrum of the chosen mesh. In this case the
modal set has to be computed explicitly.

Accordingly, the numerical investigation is organized as follows. We first
recall the two Gierer--Meinhardt benchmark regimes and the homogeneous-mode
admissibility check used throughout the section. We then analyse the IE--$\mathcal S$
schemes through parameter-only algebraic criteria, and the EX--$\mathcal S$ schemes
through fixed-mesh modal diagnostics. Finally, we introduce area-based error
metrics that summarize, for each family, the distortion of the continuous
Turing region and the resulting choice of admissible time steps.

Throughout the section we consider the Gierer--Meinhardt system
\eqref{eq:gm_model_motivating} with homogeneous Neumann boundary conditions. The homogeneous
equilibrium and the reaction Jacobian are those given in
\eqref{eq:gm_equilibrium_motivating}--\eqref{eq:gm_jacobian_motivating}. 

\subsection{Benchmark setting and homogeneous-mode admissibility}
\label{subsec:benchmarks_homogeneous}

We use the two Gierer--Meinhardt parameter sets introduced in Section~\ref{sec:motivating_examples}. The
first one (see Section \ref{subsec:motivating_imex_spurious}) is a continuous stable benchmark in which IMEX produces a
stable spurious numerical pattern. The parameters are given in \eqref{eq:imex_spurious_parameters}.
The kinetic equilibrium is stable and the continuous Turing
polynomial \( h(\mu)\) is strictly positive for all $\mu\geq 0$. Numerically,
\[
\tau^*=-4.9419,\qquad
\delta^*=6,\qquad
D_ug_v+D_vf_u=-42.8023.
\]
Hence the continuous model has no diffusion-driven instability. 

The second benchmark, discussed in Section \ref{subsec:motivating_iease_spurious_stability}, is a continuous Turing regime in which the
ASE/APE family may spuriously suppress the unstable modes. The parameters are given in \eqref{eq:iease_spurious_parameters}.
Here the homogeneous equilibrium is kinetically stable, but the continuous
Turing polynomial is negative on a non-empty band. Numerically,
\[
\tau^*=-7.5742,\qquad
\delta^*=50,\qquad
D_ug_v+D_vf_u=378.1335.
\]

For both IE--$\mathcal{S}$ and EX--$\mathcal{S}$ schemes, the homogeneous mode corresponds to
$\mu=0$. Since both diffusion treatments satisfy
\(
b_u^{\mathcal D}(0)=b_v^{\mathcal D}(0)=1,
\)
the homogeneous-mode stability thresholds depend only on the reaction
amplification matrix $A_{h_t}^{\mathcal S}$. Thus, the thresholds in
Table~\ref{tab:hom_thresholds_section7} apply to both the IE--$\mathcal{S}$ and EX--$\mathcal{S}$
versions of each reaction family.

\begin{table}[htp]
\centering
\begin{tabular}{lcc}
\toprule
Reaction family
& Example 1 
& Example 2  \\
\midrule
EE
& $7.15\times 10^{-1}$
& $1.51\times 10^{-1}$ \\
SE/PE
& $1.20\times 10^{0\phantom{x}}$
& $1.28\times 10^{-1}$ \\
ASE/APE
& $5.55\times 10^{-1}$
& $6.25\times 10^{-1}$ \\

EVSE/EVPE/EVASE/EVAPE
& $4.05\times 10^{-1}$
& $1.39\times 10^{-1}$ \\
\bottomrule
\renewcommand{\arraystretch}{0.5}
\end{tabular}
\caption{Homogeneous-mode stability thresholds $\overline h_{\rm hom}$ for
the two Gierer--Meinhardt benchmarks. These values provide only an
admissibility check at $\mu=0$ and do not determine whether the numerical
scheme correctly detects the non-homogeneous Turing modes.}
\label{tab:hom_thresholds_section7}
\end{table}

These thresholds are necessary only to guarantee that the homogeneous
equilibrium is stable for the numerical map. They do not control the behaviour of non-homogeneous modes. This distinction is already visible in the two motivating examples. In Example~1, the time step $h_t^{(1)}=0.67$ lies below the
homogeneous thresholds of the EE and SE/PE families. Hence, at this time step,
IE--EE and IE--SE/PE preserve the stability of the homogeneous mode. However,
the same value of $h_t$ is not admissible for the ASE/APE and EV-type
families, which require smaller time steps in order to satisfy the
homogeneous-mode stability conditions. Therefore, the spurious IMEX pattern observed in \ref{subsec:motivating_imex_spurious} is not due to a
loss of homogeneous stability: IE--EE is stable at $\mu=0$, but loses
non-homogeneous stability through the second Jury condition. 
In Example~2, the chosen value $h_t^{(2)}=5\times 10^{-2}$ lies below the
homogeneous thresholds of all the families considered. Nevertheless, this does
not guarantee faithful pattern detection, as shown in \ref{subsec:motivating_iease_spurious_stability} for the ASE/APE schemes which 
suppress modes belonging to the continuous Turing band.

\subsection{Parameter-only diagnostics for the IE--$\mathcal S$ schemes}
\label{subsec:IE_algebraic_diagnostics}

We first consider the IE--$\mathcal{S}$ family. In view of
Proposition~\ref{prop:negative_jury_some_mu}, a fully discrete Turing
instability for an IE--$\mathcal{S}$ scheme occurs precisely at those parameters that
satisfy
\begin{equation}
\label{eq:IE_alg_criterion_72}
0<h_t<\overline{h}_{\rm hom}^{\mathcal S}
\qquad\text{and}\qquad
\exists\, i\in\{1,2\}
\ \text{such that}\
B_{\sigma_i}^{\mathcal S}<0,\ 
\Delta_{\sigma_i}^{\mathcal S}>0,
\end{equation}
where $\sigma_1=-1$ 
and $\sigma_2=+1$ identify the first and second Jury quantities $\mathcal J_1$ and $\mathcal J_2$, respectively. The first inequality is
the homogeneous-mode admissibility check at $\mu=0$, while the pair
$B_{\sigma_i}^{\mathcal S}<0$, $\Delta_{\sigma_i}^{\mathcal S}>0$ guarantees that the quadratic
$P_{\sigma_i}^{\mathcal S}(\mu;h_t)$ takes negative values for some $\mu>0$, i.e. that
the $i$-th Jury condition is violated on a non-homogeneous mode. This criterion
depends only on the reaction parameters, on the diffusion coefficients and on
the time step $h_t$, and not on the spatial mesh.



For the first benchmark, the continuous Turing set is empty because
$h(\mu)>0$ for every $\mu\geq0$. At the time step
$h_t^{(1)}=0.67$, the homogeneous mode is stable for the EE and SE/PE
families, whereas the ASE/APE and EV-type families require smaller time steps.
Therefore, the algebraic comparison at $h_t^{(1)}$ is restricted to IE--EE and
IE--SE/PE.

\begin{table}[htp]
\centering
\begin{tabular}{lrrrrl}
\toprule
Method
& $B_{\sigma_1}^{\mathcal S}$
& $\Delta_{\sigma_1}^{\mathcal S}$
& $B_{\sigma_2}^{\mathcal S}$
& $\Delta_{\sigma_2}^{\mathcal S}$
& Algebraic outcome \\
\midrule
IE--EE
& $19.214$
& $311.141$
& $-1.794$
& $1.682$
& $\mathcal J_2$ violated \\
IE--SE/PE
& $6.358$
& $22.865$
& $11.062$
& $96.902$
& no Jury violation \\
\bottomrule
\renewcommand{\arraystretch}{0.5}
\end{tabular}
\caption{Algebraic IE--$\mathcal{S}$ diagnostics for Example~1 at $h_t^{(1)}=0.67$.
The continuous model is Turing stable. The violation criterion is
$B_{\sigma_i}^{\mathcal S}<0$ and $\Delta_{\sigma_i}^{\mathcal S}>0$.}
\label{tab:IE_algebraic_example1}
\end{table}

For IMEX, namely IE--EE, the first Jury quantity preserves exactly the sign of the continuous Turing polynomial. Accordingly, in
Table~\ref{tab:IE_algebraic_example1}, $B_{\sigma_1}^{\mathrm{EE}}>0$, so the first
Jury condition is not violated. The instability is instead produced by the
second Jury condition, since
\[
B_{\sigma_2}^{\mathrm{EE}}=-1.794<0,
\qquad
\Delta_{\sigma_2}^{\mathrm{EE}}=1.682>0.
\]
The spurious pattern observed in the nonlinear simulation is therefore a
discrete $\mathcal J_2$-driven instability. By contrast, for
IE--SE/PE both $B_{\sigma_1}^{\mathrm{SE}}$ and $B_{\sigma_2}^{\mathrm{SE}}$ are positive, and
no algebraic non-homogeneous instability is detected at the same time step.

For the second benchmark, the continuous Turing interval is non-empty. At
$h_t^{(2)}=5\times10^{-2}$, the homogeneous-mode stability condition is
satisfied by all the IE--$\mathcal{S}$ families considered. Hence the algebraic diagnostic
can be applied to all schemes.

\begin{table}[htp]
\centering
\begin{tabular}{lrrrrl}
\toprule
Method
& $B_{\sigma_1}^{\mathcal S}$
& $\Delta_{\sigma_1}^{\mathcal S}$
& $B_{\sigma_2}^{\mathcal S}$
& $\Delta_{\sigma_2}^{\mathcal S}$
& Algebraic outcome \\
\midrule
IE--EE
& $-0.945$
& $0.694$
& $17.045$
& $285.155$
& $\mathcal J_1$ violated \\
IE--SE/PE
& $-1.069$
& $0.915$
& $17.169$
& $289.971$
& $\mathcal J_1$ violated \\
IE--EVSE/EVPE
& $-0.936$
& $0.676$
& $17.036$
& $285.433$
& $\mathcal J_1$ violated \\
IE--ASE/APE
& $0.036$
& $-0.132$
& $16.064$
& $252.579$
& no Jury violation \\
IE--EVASE/EVAPE
& $0.540$
& $0.091$
& $15.560$
& $237.325$
& no Jury violation \\
\bottomrule
\renewcommand{\arraystretch}{0.5}
\end{tabular}
\caption{Algebraic IE--$\mathcal{S}$ diagnostics for Example~2 at
$h_t^{(2)}=5\times10^{-2}$. The continuous model is Turing unstable. The
violation criterion is $B_{\sigma_i}^{\mathcal S}<0$ and
$\Delta_{\sigma_i}^{\mathcal S}>0$.}
\label{tab:IE_algebraic_example2}
\end{table}

In this case, IMEX detects the continuous instability through the first Jury
condition, since
\[
B_{\sigma_1}^{\mathrm{EE}}=-0.945<0,
\qquad
\Delta_{\sigma_1}^{\mathrm{EE}}=0.694>0.
\]
The same $\mathcal J_1$-driven detection is observed for IE--SE/PE and
IE--EVSE/EVPE. By contrast, the IE--ASE/APE family does not satisfy the
algebraic violation criterion:
\[
B_{\sigma_1}^{\mathrm{ASE}}=0.036>0,
\qquad
\Delta_{\sigma_1}^{\mathrm{ASE}}=-0.132<0,
\qquad
B_{\sigma_2}^{\mathrm{ASE}}=16.064>0.
\]
Thus, the corresponding Jury quantities remain positive, and the scheme may
suppress modes belonging to the continuous Turing band. The same qualitative
behaviour is found for IE--EVASE/EVAPE, for which
$B_{\sigma_1}^{\mathrm{EVASE}}>0$ and $B_{\sigma_2}^{\mathrm{EVASE}}>0$. 

\subsubsection{IE--$\mathcal S$ parameter-plane maps}
\label{subsubsec:IE_region_maps}

We now use an algebraic characterization of the fully discrete Turing region to classify the IE--$\mathcal{S}$ schemes in the
$(a,b)$ parameter plane. Let $\mathcal C$ denote the continuous Turing region,
namely the set of parameter pairs for which the homogeneous equilibrium is
kinetically stable and $h(\mu)<0$ for some $\mu>0$.

For the Gierer--Meinhardt parametrization used here, the continuous
boundaries can be computed from \eqref{eq:algebraic_continum}. In the region
maps, we display the Hopf boundary, corresponding to $\tau^*=0$,
\[
a_H=\frac{b(c-b)}{b+c},
\]
the determinant boundary, denoted by $a_D$, where $\delta^*=0$, and the
continuous Turing boundary $a_T$, obtained from
\[
m=2\sqrt{D_uD_v\delta^*},
\qquad
m:=D_ug_v+D_vf_u.
\]

For a fixed IE--$\mathcal{S}$ scheme and a fixed time step $h_t$, the parameter-plane
classification is based on the comparison between $\mathcal C$ and the corresponding algebraic fully discrete Turing region, defined as
\begin{equation}
\label{eq:D_alg_IE_maps_section7}
\mathcal D_{\rm alg}^{IE-S}(h_t)
=
\left\{
(a,b):
0<h_t<\overline{h}_{\rm hom}^{\mathcal S}
\ \text{and}\
\exists i\in\{1,2\}
\ \text{such that}\
B_{\sigma_i}^{\mathcal S}<0,\ 
\Delta_{\sigma_i}^{\mathcal S}>0
\right\}.
\end{equation}
The region \(\mathcal D_{\rm alg}^{\mathrm{IE}-\mathcal{S}}(h_t)\) is obtained only from the
model parameters, the diffusion coefficients and the time step. Therefore no
computation of the discrete modal values $\mu_{ij}$ is required for the IE--$\mathcal{S}$
classification.

Since for the IE--$\mathcal S$ family $h_t$ enters the algebraic region
\eqref{eq:D_alg_IE_maps_section7} as an additional parameter, the deformation of
the fully discrete Turing region relative to the fixed continuous Turing set
$\mathcal C$ is displayed, for both benchmarks, in
Figures~\ref{fig:IE_maps_example1_all}--\ref{fig:IE_maps_example2_all}, where each
row is read across decreasing time steps.
We follow the colour convention in Table~\ref{tab:colour_convention}. The white regions are not included in the Turing-region comparison, since they correspond to parameter values for which the selected fully discrete scheme is already inadmissible at the homogeneous mode. Thus, white points should not be interpreted as delayed detection of a continuous Turing instability, but as a loss of homogeneous-mode stability for the chosen method and time step.


\begin{table}[htp]
\centering
\small
\renewcommand{\arraystretch}{1.15}
\begin{tabular}{p{0.11\textwidth}p{0.19\textwidth}p{0.19\textwidth}p{0.36\textwidth}}
\toprule
\textbf{Colour} & \textbf{Continuous regime} & \textbf{Discrete regime} & \textbf{Meaning} \\
\midrule
Green  & Turing     & Turing     & Correct detection \\
Orange & non-Turing & Turing     & Spurious discrete instability \\
Blue   & Turing     & non-Turing & Suppression of a continuous instability \\
Grey   & non-Turing & non-Turing & Correct agreement in the non-Turing regime \\
White  & --         & --         & Homogeneous-mode inadmissibility for the selected scheme and time step \\
\bottomrule
\renewcommand{\arraystretch}{0.5cm}
\end{tabular}
\caption{Colour convention for the parameter-plane region maps. The classification compares the continuous Turing region with the fully discrete Turing region, while separately marking the loss of admissibility of the homogeneous discrete mode.}
\label{tab:colour_convention}
\end{table}

Around the benchmark of Example~1, corresponding to
\[
D_u=1,\qquad D_v=12,\qquad c=0.4,\qquad \gamma=10,
\qquad
(a,b)=(a_1,b_1)=(2,0.15),
\]
we consider all the IE--$\mathcal{S}$ families and track how their fully discrete Turing regions evolve as the time step decreases. 

For larger time steps, the homogeneous-mode stability condition $0<h_t<\overline{h}_{\rm{hom}}^{\mathcal{S}}$ is not satisfied near the benchmark point for the ASE/APE and EV-type families. These schemes therefore become admissible only for smaller time steps.

The corresponding region maps are collected in
Figure~\ref{fig:IE_maps_example1_all}, where the rows show the five IE--$\mathcal{S}$ families (IE--EE, IE--SE/PE, IE--EVSE/EVPE, IE--ASE/APE and IE--EVASE/EVAPE), and the columns correspond to the
decreasing time steps $h_t=0.67,\,0.5,\,0.4,\,0.1,\,0.05$, with the first column corresponding to $h_t^{(1)}=0.67$.
In each panel, the black curve represents
the continuous Turing boundary, the red curve the fully discrete Turing
boundary $\partial\mathcal D_{\rm {alg}}^{\rm{IE}-\mathcal S}(h_t)$, the crosses mark the
continuous Turing region $\mathcal C$, and the yellow dot
locates the parameters $(a_1,b_1)=(2,0.15)$. Because the relevant
features lie at large $a$ for large $h_t$ and shift toward smaller $a$ as
$h_t$ decreases, the $a$-window is widened in the first two columns and
progressively zoomed in the remaining ones, while the $b$-window is kept fixed.

At the  value $h_t^{(1)}=0.67$ (first column), the different IE--$\mathcal{S}$ families exhibit distinct behaviours. The IE--EE (IMEX) row shows an extended orange region surrounding the  point \((a_1,b_1)\). Since the
continuous model is Turing-stable there, this orange region is not a continuous
Turing effect: it is classified as spurious
(in Table~\ref{tab:colour_convention}) and corresponds to the $\mathcal J_2$-driven stable spurious
pattern observed in the nonlinear simulation of
Section~\ref{sec:motivating_examples}. 

By contrast, the IE--SE/PE row shows
no such artificial non-homogeneous instability at the point \( (a_1,b_1)\), in
agreement with the results of
Table~\ref{tab:IE_algebraic_example1}, where both
$B_{\sigma_1}^{\mathrm{SE}}$ and $B_{\sigma_2}^{\mathrm{SE}}$ are positive. As $h_t$ decreases
along each row, the spurious orange region of IE--EE contracts, and the red
discrete boundary approaches the black continuous one, so that both schemes
progressively recover the thin continuous Turing sliver (green). This confirms
that the spurious $\mathcal J_2$ instability seen at $h_t^{(1)}=0.67$ is a
large-step effect that is removed by refinement, and that the continuous and
fully discrete Turing predictions coincide in the limit of small $h_t$.

The behaviour of the remaining IE--$\mathcal{S}$ families can be described within the same framework. At the larger time steps $h_t=0.67$ and $h_t=0.5$, the homogeneous-mode stability condition 
$0<h_t<\overline{h}_{\rm hom}^{\mathcal S}$ fails near the point $(a_1,b_1)$ for the 
ASE/APE and EV-type families. Accordingly, the corresponding cells are left white, 
indicating homogeneous-mode inadmissibility, and these schemes become meaningful 
only at smaller time steps.

Once admissible, the EE-type and SE/PE-type rows exhibit the same 
anticipatory-then-recovering behaviour described above, characterized by the 
appearance of a spurious (orange) region at larger time steps and its 
progressive contraction under refinement.

By contrast, the adjoint-symplectic families (ASE/APE and EVASE/EVAPE) display, 
on this Turing-stable benchmark, only the residual delayed (blue) structure 
associated with their algebraic regions, without producing the strong 
anticipatory (orange) destabilization observed for IE--EE.

\begin{figure}[t]
\centering
\includegraphics[width=\textwidth]
{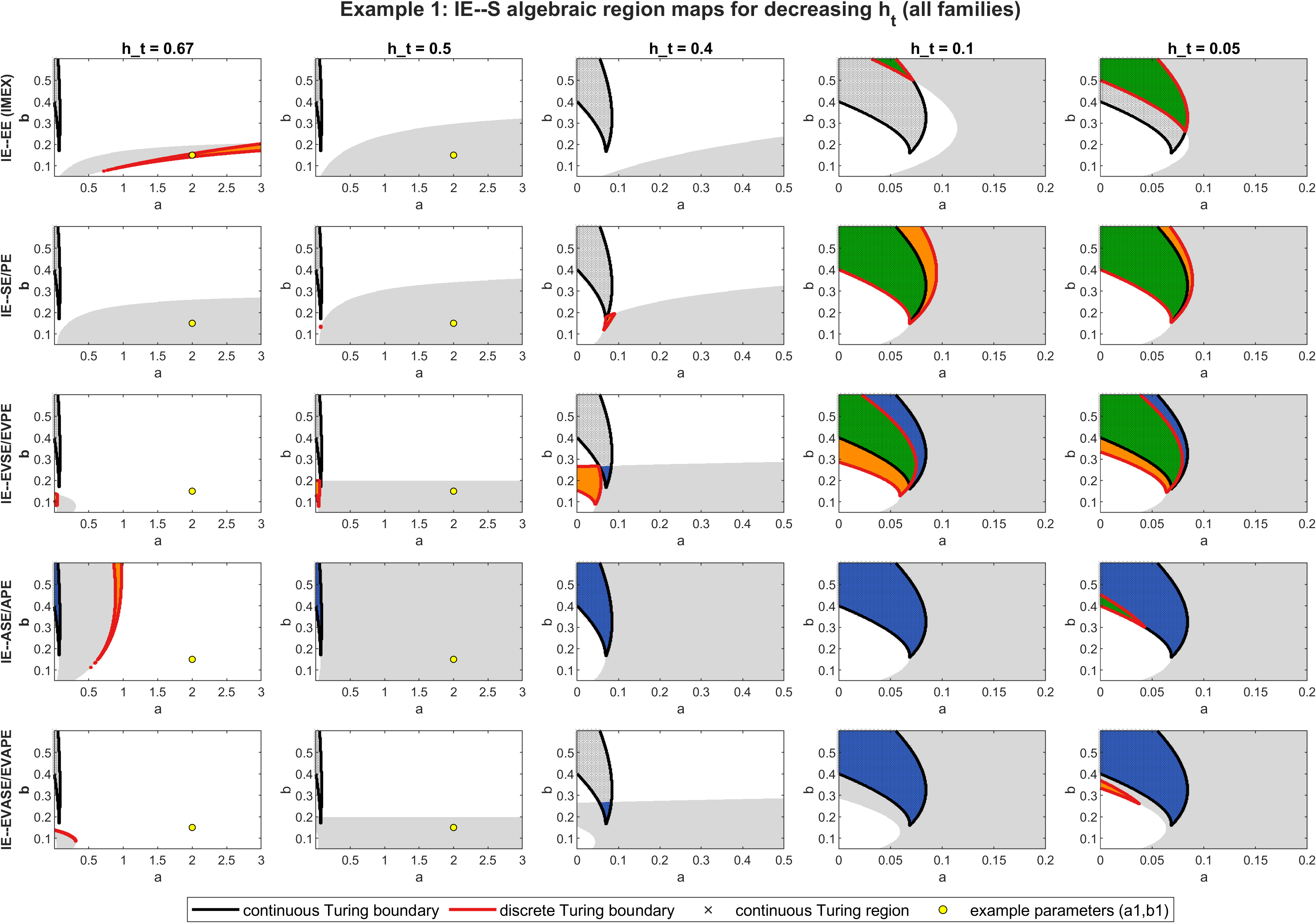}
\caption{IE--$\mathcal{S}$ algebraic region maps in the neighbourhood of Example~1 for all
five families, arranged as a grid: the rows are IE--EE (IMEX), IE--SE/PE,
IE--EVSE/EVPE, IE--ASE/APE and IE--EVASE/EVAPE, while the columns correspond to
the decreasing time steps $h_t=0.67,\,0.5,\,0.4,\,0.1,\,0.05$. 
Colours and boundaries follow the convention of Table~\ref{tab:colour_convention}: green denotes correct Turing
detection, orange a spurious discrete instability, blue a suppressed instability, and grey agreement in the non-Turing regime. The black and red curves are, respectively, the continuous and the fully discrete Turing boundaries, and the yellow dot marks $(a_1,b_1)=(2,0.15)$. The $a$-window is wide in the first two columns and progressively zoomed in the remaining ones, so that the extended orange region of IE--EE at $h_t^{(1)}=0.67$ remains visible while the small-$a$ behaviour at small $h_t$ is resolved; the $b$-window is fixed. White cells mark homogeneous-mode inadmissibility of the selected scheme at the corresponding time step, which affects the ASE/APE and EV-type rows at the largest time steps. As $h_t$ decreases, the admissible cells show the fully discrete Turing boundary approaching the continuous one.}
\label{fig:IE_maps_example1_all}
\end{figure}

Around the benchmark of Example~2, corresponding to
\[
D_u=1,\qquad D_v=160,\qquad c=1,\qquad \gamma=5,
\qquad
(a, b) = (a_2,b_2)=(0.6933,2),
\]
we start from the  time step $h_t^{(2)}=5\times10^{-2}$ and, as for
Example~1, follow the schemes as the time step is varied. At this benchmark the
homogeneous-mode stability condition is satisfied by all the IE--$\mathcal S$ families
considered, so any of them can be displayed directly. Here the continuous
Turing region is non-empty and is delimited by the black boundary running close
to the point \( (a_2, b_2)\) . 

Figure~\ref{fig:IE_maps_example2_all} reports the by-method,
decreasing-$h_t$ layout for all five IE--$\mathcal S$ representatives, with
\(
h_t=0.2,\;0.15,\;0.1,\;0.05,
\)
the last column corresponding to the benchmark value \(h_t^{(2)}\). The rows
reveal two distinct behaviours. The IE--EE, IE--SE/PE and IE--EVSE/EVPE
families progressively recover the continuous Turing band: the green region,
where the continuous diffusion-driven instability is correctly detected at the
fully discrete level, expands as \(h_t\) decreases, and the red discrete Turing
boundary moves towards the black continuous one. For the coarser steps,
especially \(h_t=0.2\) and \(h_t=0.15\), this agreement is only partial in some
rows, showing a visible time-step distortion of the parameter-plane
classification. Under refinement, however, the discrepancy is strongly reduced,
and for \(h_t=0.1\) and \(h_t=0.05\) these three families essentially reproduce
the continuous Turing region in the plotted window.

The adjoint-symplectic rows, IE--ASE/APE and IE--EVASE/EVAPE, display the
opposite qualitative behaviour. A large part of the continuous Turing band is
coloured in blue, meaning that the continuous model is Turing unstable whereas
the corresponding fully discrete scheme remains stable. This is the
parameter-plane manifestation of the delayed error already detected by the
algebraic indicators: the pattern-forming modes are not activated by the
discrete dynamics, and the continuous Turing instability is therefore artificially
suppressed. The effect is clearest for IE--ASE/APE, where the blue region covers
almost the whole continuous Turing band for the displayed time steps. For
IE--EVASE/EVAPE the geometry is slightly more mixed at the coarser steps, but
the same delayed-stabilization mechanism dominates the neighbourhood of the
benchmark. Thus the comparison confirms that the discrepancy is not a
single-step accident, but a systematic family-dependent effect controlled by the
time discretization.

\begin{figure}[t]
\centering
\includegraphics[width=\textwidth]{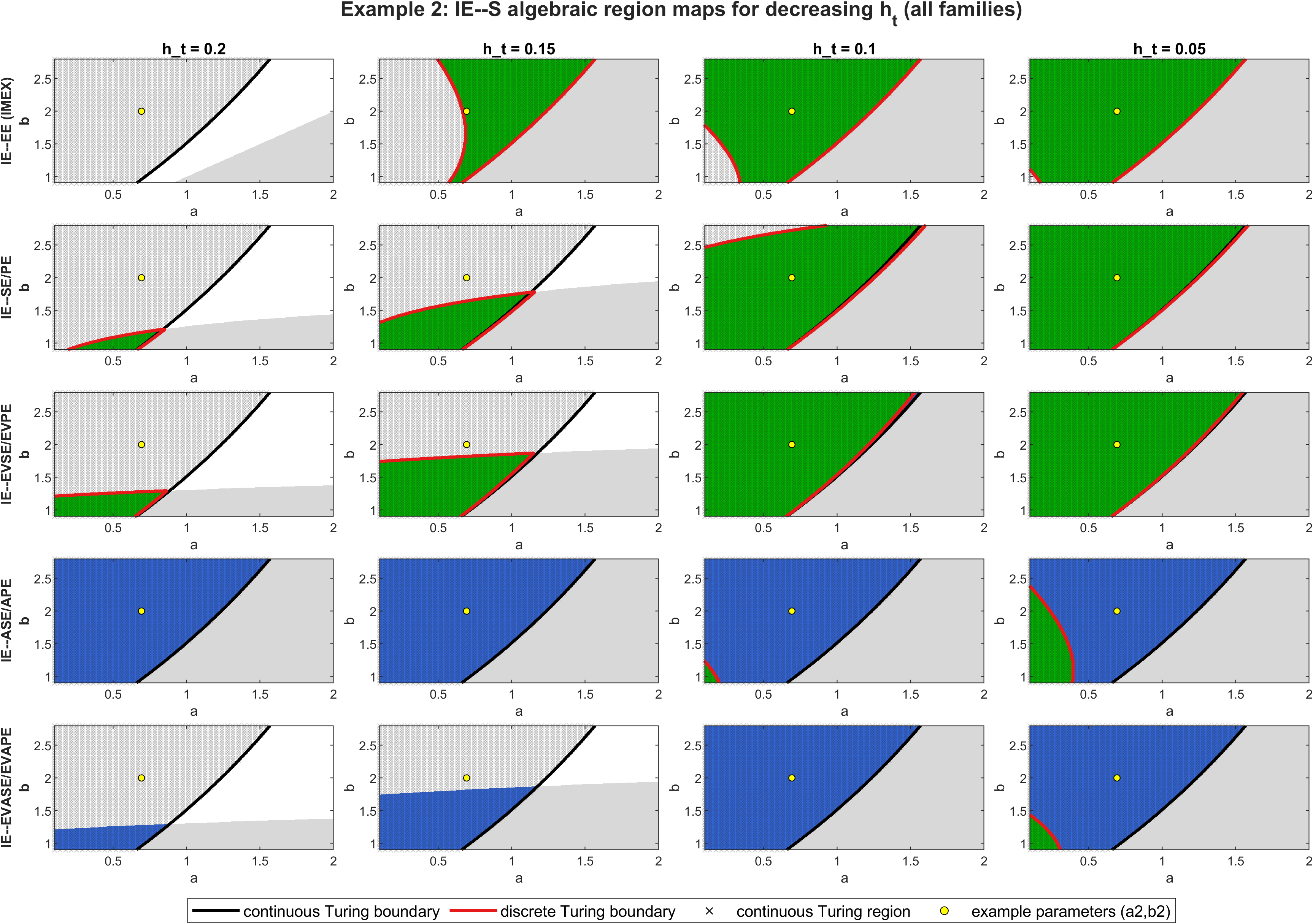}
\caption{IE--$\mathcal{S}$ algebraic region maps in the neighbourhood of Example~2 for all
five families. The rows are IE--EE (IMEX), IE--SE/PE, IE--EVSE/EVPE,
IE--ASE/APE and IE--EVASE/EVAPE, while the columns correspond to the decreasing
time steps $h_t=0.2,\,0.15,\,0.1,\,0.05$. The $(a,b)$-window is fixed across all
panels and colours follow the convention of
Table~\ref{tab:colour_convention}. The IE--EE, IE--SE/PE and
IE--EVSE/EVPE rows progressively recover the continuous Turing band: as
$h_t$ decreases, the green region expands and the red fully discrete Turing
boundary approaches the black continuous one. The IE--ASE/APE and
IE--EVASE/EVAPE rows instead show a delayed error, with large blue regions
inside the continuous Turing band, where the continuous model is Turing unstable
but the fully discrete scheme remains stable. This highlights the
family-dependent spurious stabilization produced by the adjoint-symplectic
variants.}
\label{fig:IE_maps_example2_all}
\end{figure}
\subsubsection{Summary for the IE--$\mathcal S$ family}
\label{subsubsec:IE_region_maps_htdependence}

The parameter-plane maps reveal two complementary, mutually exclusive failure
modes, each tied to a distinct Jury condition.

In the Turing-stable benchmark (Example~1), the only admissible representatives
at $h_t^{(1)}=0.67$ are IE--EE and IE--SE/PE. IE--EE develops an extended
spurious instability region around $(a_1,b_1)$, driven by the violation of the
second Jury condition ($B^{EE}_{\sigma_2}<0$, $\Delta^{EE}_{\sigma_2}>0$), while
$\mathcal J_1$ correctly reproduces the empty continuous Turing set. As $h_t$ is
refined this $\mathcal J_2$-driven region contracts monotonically and the
discrete boundary $\partial\mathcal{D}^{IE-\mathcal S}_{\mathrm{alg}}(h_t)$
collapses onto the continuous one: the spurious pattern is therefore a
large-step artefact, not a structural defect. IE--SE/PE shows no such region at
any admissible step.

In the continuous  Turing regime (Example~2), the roles separate by reaction family
rather than by time step. The EE, SE/PE and EVSE/EVPE rows track the continuous
Turing boundary through a $\mathcal J_1$ violation, the green region expanding
under refinement. The adjoint-symplectic rows IE--ASE/APE and IE--EVASE/EVAPE
instead leave the continuous band positive ($B^{ASE}_{\sigma_1}>0$,
$\Delta^{ASE}_{\sigma_1}<0$): they suppress the pattern-forming modes, producing
the blue region. Crucially, this suppression is not removed by
ordinary refinement at the same rate as the IE--EE artefact; it shrinks only as
$h_t$ becomes small enough to bring $B^{ASE}_{\sigma_1}$ below zero, so it is a
systematic, family-dependent effect controlled by the reaction substep, not a
single-step accident.

\subsection{Modal diagnostics for the EX--$\mathcal S$ schemes}
\label{subsec:EX_modal_diagnostics}

We now turn to the EX--$\mathcal S$ family. In this case the Jury functions have the form
\begin{equation}
\label{eq:EX_Jury_section7}
\mathcal J_i^{EX-\mathcal S}(\mu;h_t)
=
1+\sigma_i a_{11}^S e^{-h_tD_u\mu}
+\sigma_i a_{22}^S e^{-h_tD_v\mu}
+\det(A_{h_t}^S)e^{-h_t(D_u+D_v)\mu},
\qquad i=1,2.
\end{equation}
Equivalently, setting
\[
x=e^{-h_tD_u\mu},
\qquad
r=\frac{D_v}{D_u},
\]
one obtains the generalized polynomial
\begin{equation}
\label{eq:EX_generalized_polynomial_section7}
F_i^S(x)
=
1+\sigma_i a_{11}^S x
+\sigma_i a_{22}^S x^r
+\det(A_{h_t}^S)x^{1+r},
\qquad x\in(0,1).
\end{equation}
When $r$ is an integer, this is a polynomial in $x$ of degree $1+r$; otherwise
it is a generalized polynomial. In either case, and especially for large
diffusion contrasts, no simple mode-independent criterion comparable with
\eqref{eq:IE_alg_criterion_72} is obtained.

Therefore, for EX--$\mathcal{S}$ schemes, the discrete modal spectrum has to be computed.
Using the tensor-product finite-difference semi-discretization with
homogeneous Neumann boundary conditions, the one-dimensional discrete
Laplacians have eigenvalues $\lambda_i^x$ and $\lambda_j^y$, so that the
two-dimensional modal parameters $\mu_{ij}=-(\lambda_i^x+\lambda_j^y)\geq0$ of
\eqref{eq:mu_modes} are now evaluated explicitly on the chosen mesh.
We denote by
\(
\Lambda_h=\{\mu_{ij}>0\}
\)
the set of positive discrete spatial modes on the fixed mesh, and by
$\Lambda_h^M$ the retained set of the first $M$ positive modes.

For a fixed EX--$\mathcal{S}$ scheme, a fixed mesh and a fixed time step, we define the
fully discrete modal instability region as
\begin{equation}
\label{eq:D_EX_modal_section7}
\mathcal D_{h,M}^{EX-\mathcal S}(h_t)
=
\left\{
(a,b):
0<h_t<\overline h_{\rm hom}^{S}
\ \text{and}\
\exists \mu\in\Lambda_h^M
\ \text{such that}\
\mathcal J_1^{EX-\mathcal S}(\mu;h_t)<0
\ \text{or}\
\mathcal J_2^{EX-\mathcal S}(\mu;h_t)<0
\right\}.
\end{equation}
Unlike \eqref{eq:D_alg_IE_maps_section7}, the set
$\mathcal D_{h,M}^{EX-\mathcal S}$ depends on the spatial mesh and on the retained
modal truncation.

\subsubsection{EX--$\mathcal{S}$ modal margins at the two motivating benchmarks}
\label{subsubsec:EX_margins_benchmarks}

For the EX--$\mathcal{S}$ family the stability test cannot be reduced to a
parameter-only algebraic condition analogous to
\eqref{eq:D_alg_IE_maps_section7}. Indeed, the Jury quantities are evaluated
through the generalized functions $F_i^\mathcal S$ in
\eqref{eq:EX_generalized_polynomial_section7}, with
\[
x_\mu=e^{-h_tD_u\mu},
\qquad
\mu\in\Lambda_h^M.
\]
Hence the diagnostic depends on the retained modal set. In the computations
below we use the same finite-difference mesh as in the nonlinear benchmarks,
with $h=0.1$, and retain the first $M=500$ positive Neumann modes.
The five linearly distinct EX--$\mathcal S$ representatives are displayed in the order
EX--EE, EX--SE/PE, EX--EVSE/EVPE, EX--ASE/APE and EX--EVASE/EVAPE.

For the first benchmark the continuous Turing set is empty. It is therefore
not meaningful to restrict the first margin to a continuous Turing band. We
use instead the all-mode diagnostics
\begin{equation}
\label{eq:EX_margins_example1}
m_{1,\mathrm{all}}^{EX-\mathcal{S}}(h_t)
=
\min_{\mu\in\Lambda_h^M}
F_1^\mathcal S(x_\mu),
\qquad
m_2^{EX-\mathcal S}(h_t)
=
\min_{\mu\in\Lambda_h^M}
F_2^\mathcal S(x_\mu).
\end{equation}
A negative value of either margin signals that the fully discrete EX--$\mathcal S$ map
amplifies some retained non-homogeneous mode, although the continuous problem
is Turing-stable.

Figure~\ref{fig:EX_Fi_example1} shows the two margins as functions of the time
step. The marker on the horizontal axis indicates the value
$h_t^{(1)}=0.67$, while the coloured dashed vertical lines indicate the
method-dependent homogeneous-mode thresholds. The EX--EE curve loses
non-homogeneous stability before the SE/PE-type curve, consistently with the
unbounded transient behaviour already observed for EX--EE in
Example~1. By contrast, EX--SE/PE keeps a substantially larger admissible
window and does not display the same early spurious destabilization. The
explicit-variant SE/PE family behaves close to the SE/PE representative,
whereas the adjoint families are more constrained by homogeneous-mode
admissibility. Thus, in the Turing-stable benchmark, the most favourable
EX--$\mathcal{S}$ behaviour is provided by EX--SE/PE: it delays the onset of artificial
fully discrete instability and preserves the homogeneous equilibrium over a
larger range of time steps than EX--EE and the adjoint variants.

\begin{figure}[htp]
\centering
\includegraphics[width=\textwidth]{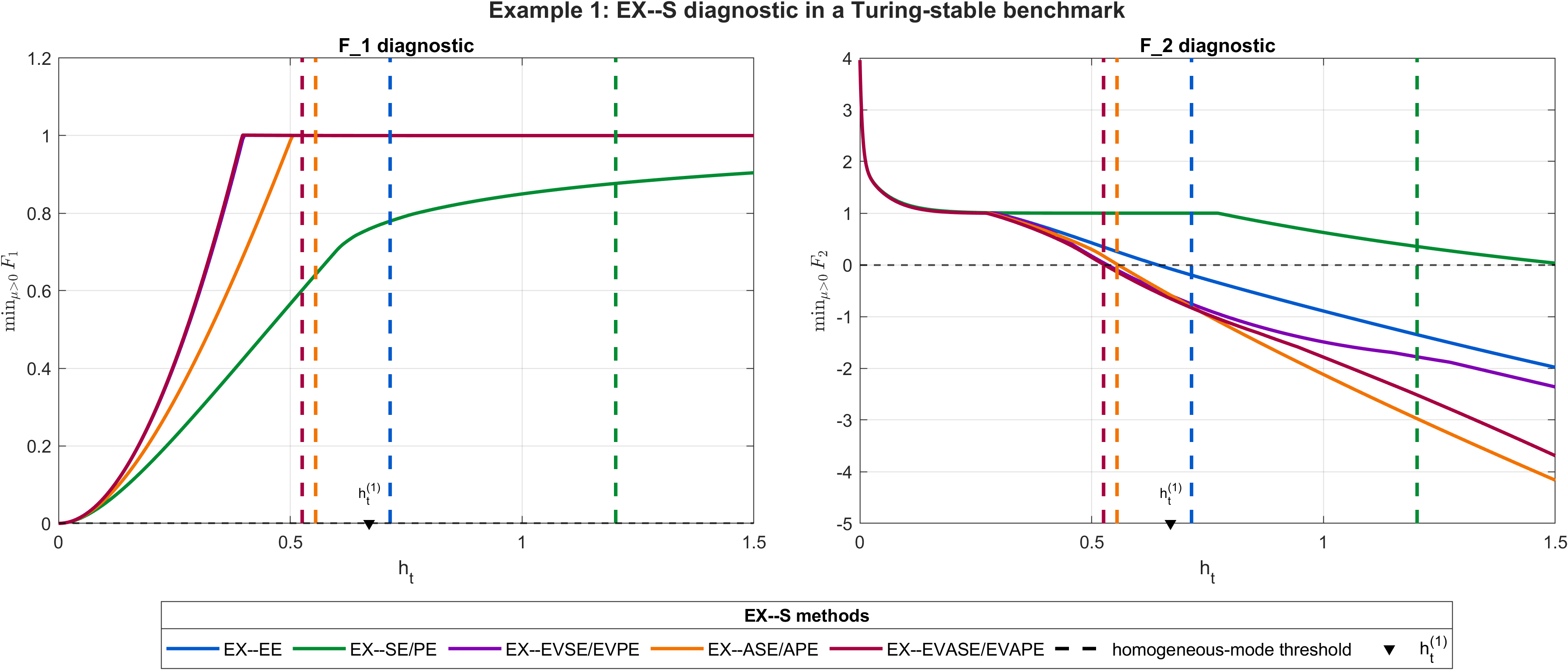}
\caption{EX--$\mathcal{S}$ modal diagnostics for Example~1. Since the continuous model is
Turing-stable, the left panel reports the all-mode margin
$m_{1,\mathrm{all}}^{EX-\mathcal{S}}$ and the right panel reports $m_2^{EX-\mathcal{S}}$. The
horizontal dashed line marks the zero level. A negative value indicates a
fully discrete non-homogeneous instability on the retained modes. The small
black marker on the horizontal axis identifies $h_t^{(1)}$, while the coloured
dashed vertical lines are the homogeneous-mode thresholds of the individual
EX--$\mathcal{S}$ schemes. EX--SE/PE has the most robust behaviour in this benchmark,
because it avoids the early spurious destabilization observed for EX--EE and
retains a larger homogeneous admissibility interval than the adjoint-type
families.}
\label{fig:EX_Fi_example1}
\end{figure}

For the second benchmark the continuous Turing interval is non-empty. We
therefore introduce the retained continuous Turing set
\[
\mathcal T_h^M
=
\{\mu\in\Lambda_h^M:h(\mu)<0\},
\]
and use the margins
\begin{equation}
\label{eq:EX_margins_example2}
m_{1,T}^{EX-\mathcal S}(h_t)
=
\min_{\mu\in\mathcal T_h^M}
F_1^S(x_\mu),
\qquad
m_2^{EX-\mathcal S}(h_t)
=
\min_{\mu\in\Lambda_h^M}
F_2^S(x_\mu).
\end{equation}
A negative value of $m_{1,T}^{EX-\mathcal S}$ means that the EX--$\mathcal S$ scheme detects the
continuous Turing mechanism on the retained Turing modes. A positive value
means that the scheme suppresses all retained modes that the continuous model
would amplify.

The behaviour in Figure~\ref{fig:EX_Fi_example2} separates the five EX--$\mathcal{S}$
families into two groups. The EX--EE, EX--SE/PE and EX--EVSE/EVPE curves of
$m_{1,T}^{EX-\mathcal{S}}$ become negative near the  step
$h_t^{(2)}=5\times10^{-2}$, and therefore detect the continuous
Turing instability. On the other hand, EX--ASE/APE and EX--EVASE/EVAPE remain
positive in the same range of time steps, showing that the adjoint families
can suppress the continuous Turing band also when the diffusive subflow is
treated exactly. The second margin remains positive in the plotted range, so
the relevant mechanism in Example~2 is not a discrete $F_2$-driven
instability, but the failure of the first EX--$\mathcal{S}$ margin to become negative on
the retained Turing modes. This confirms that the spurious stabilization of
the adjoint-symplectic family is mainly encoded in the reaction amplification
matrix and is not an artefact of implicit Euler diffusion.

\begin{figure}[htp]
\centering
\includegraphics[width=\textwidth]{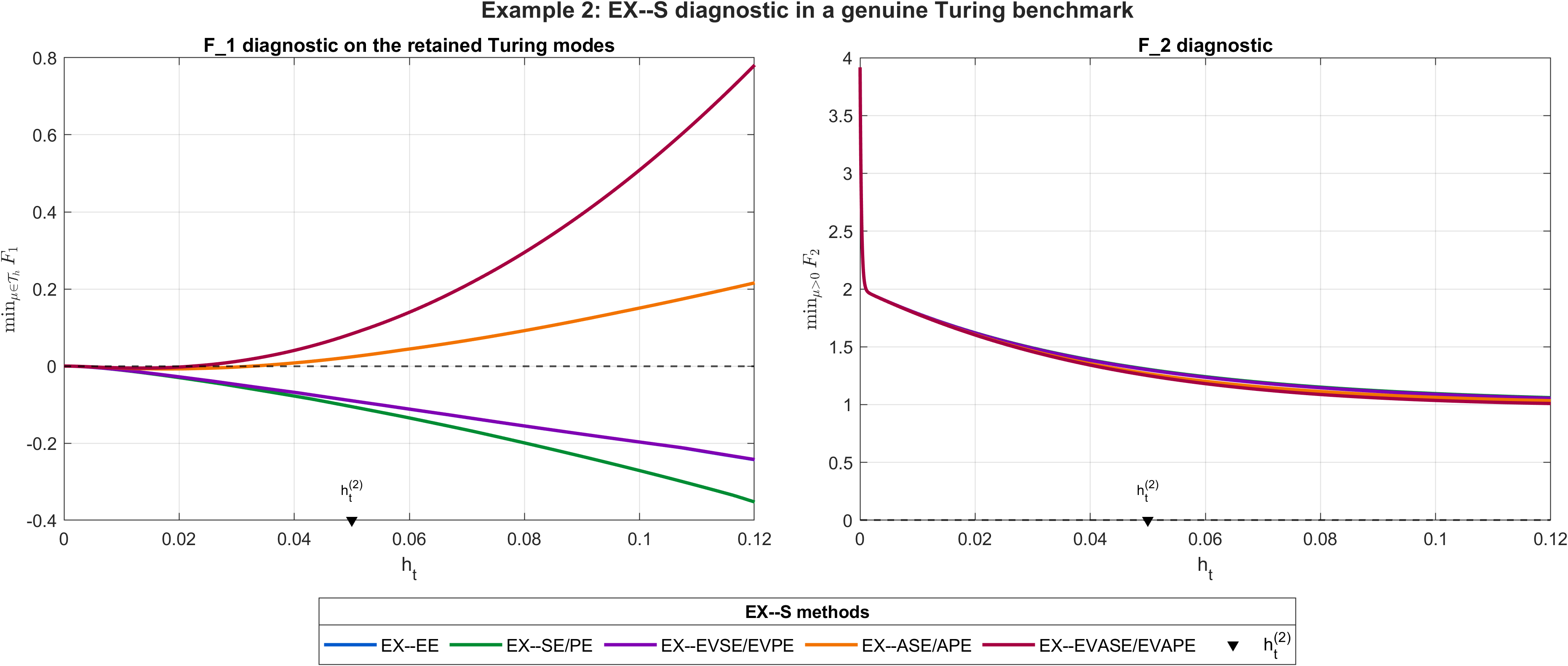}
\caption{EX--$\mathcal{S}$ modal diagnostics for Example~2. The left panel reports the
margin $m_{1,T}^{EX-\mathcal S}$ restricted to the retained continuous Turing modes,
while the right panel reports the all-mode margin $m_2^{EX-\mathcal S}$. The marker on
the horizontal axis identifies $h_t^{(2)}$. The EX--EE, EX--SE/PE and
EX--EVSE/EVPE curves become negative in the first panel and therefore detect
the continuous Turing instability. The EX--ASE/APE and EX--EVASE/EVAPE curves
remain positive, indicating spurious stabilization of the continuous Turing
band. Since the second margin stays positive, the observed distortion is a
suppression of the continuous $F_1$ mechanism rather than an $F_2$-driven
spurious instability.}
\label{fig:EX_Fi_example2}
\end{figure}

The two benchmarks lead to the same qualitative conclusion as the IE--$\mathcal{S}$
analysis. Among the EX--$\mathcal{S}$ schemes, the SE/PE representative gives the most
balanced response: in the Turing-stable benchmark it avoids the early
spurious destabilization of EX--EE, while in the Turing benchmark it
detects the unstable band instead of suppressing it as the adjoint families
do. The EVSE/EVPE variant is close to this behaviour in Example~2, but the
SE/PE scheme remains the cleaner choice because it combines correct detection
with a more robust behaviour in the Turing-stable test. Therefore, for the two
benchmarks considered here, the SE/PE family is the preferable first-order
reaction solver, both with implicit Euler diffusion and with exact diffusion.

\subsubsection{EX--$\mathcal{S}$ parameter-plane maps}
\label{subsubsec:EX_region_maps}

The EX--$\mathcal{S}$ parameter-plane maps are computed from
\eqref{eq:D_EX_modal_section7}. In contrast with the IE--$\mathcal{S}$ maps, the
classification is mesh-dependent. In the computations we use a fixed mesh
$h=0.1$ and retain the first $M=500$ positive two-dimensional Neumann
modes.

Around Example~1, the by-method maps are collected in
Figure~\ref{fig:EX_maps_example1}, with the five EX--$\mathcal{S}$ families on the rows and
the decreasing time steps $h_t=0.67,\,0.5,\,0.4,\,0.1,\,0.05$ on the columns,
using the same per-column $a$-window as the IE--$\mathcal{S}$ Figure~\ref{fig:IE_maps_example1_all}.
The EX--EE scheme may still produce red regions associated with spurious
non-homogeneous instability; however, the exact treatment of the diffusive flow
changes the modal damping factors and therefore modifies the extent of the
unstable regions with respect to the IE--EE case. As $h_t$ decreases the
admissible cells contract the spurious red region and recover the thin
continuous Turing sliver, mirroring the IE--$\mathcal{S}$ behaviour but with a mesh-dependent
classification.

Around Example~2, the by-method maps are collected in
Figure~\ref{fig:EX_maps_example2}, with the same row/column organization and the
fixed $(a,b)$-window of the IE--$\mathcal{S}$ Figure~\ref{fig:IE_maps_example2_all}. The
EX--ASE/APE and EX--EVASE/EVAPE families exhibit blue regions inside the
continuous Turing set, similarly to the corresponding IE--ASE/APE schemes. This
means that the spurious stabilization mechanism is not only a consequence of
implicit Euler diffusion, but is already encoded in the adjoint reaction
amplification matrix combined with the discrete Jury conditions; as $h_t$
decreases, the blue suppression region shrinks, confirming that the artefact is
attenuated under refinement.
\begin{figure}[htp]
\centering
\includegraphics[width=\textwidth]{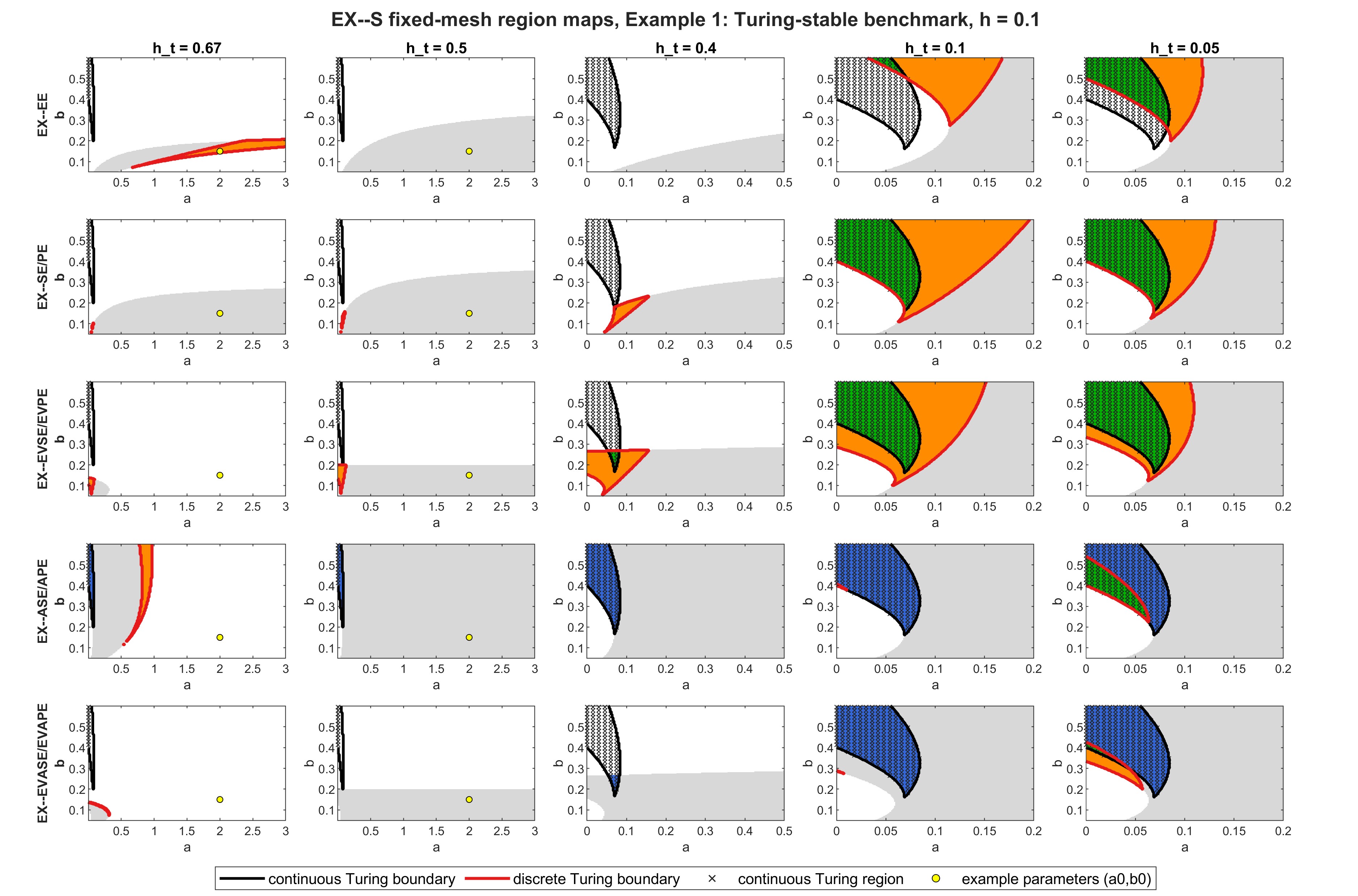}
\caption{EX--$\mathcal{S}$ fixed-mesh ($h=0.1$) algebraic region maps in the
neighbourhood of Example~1 for all five families, arranged as a grid: the rows
are EX--EE, EX--SE/PE, EX--EVSE/EVPE, EX--ASE/APE and EX--EVASE/EVAPE, while the
columns correspond to the decreasing time steps $h_t=0.67,\,0.5,\,0.4,\,0.1,\,0.05$.
Colours and boundaries follow the convention of
Table~\ref{tab:colour_convention}; the $a$-window is wide in the first
two columns and progressively zoomed in the remaining ones, while the $b$-window
is fixed, as in Figure~\ref{fig:IE_maps_example1_all}. White cells mark
homogeneous-mode inadmissibility of the selected scheme at the corresponding
time step. As $h_t$ decreases, the spurious red region of the EE-type and
SE/PE-type rows contracts and the fully discrete Turing boundary approaches the
continuous one.}
\label{fig:EX_maps_example1}
\end{figure}

\begin{figure}[htp]
\centering
\includegraphics[width=\textwidth]{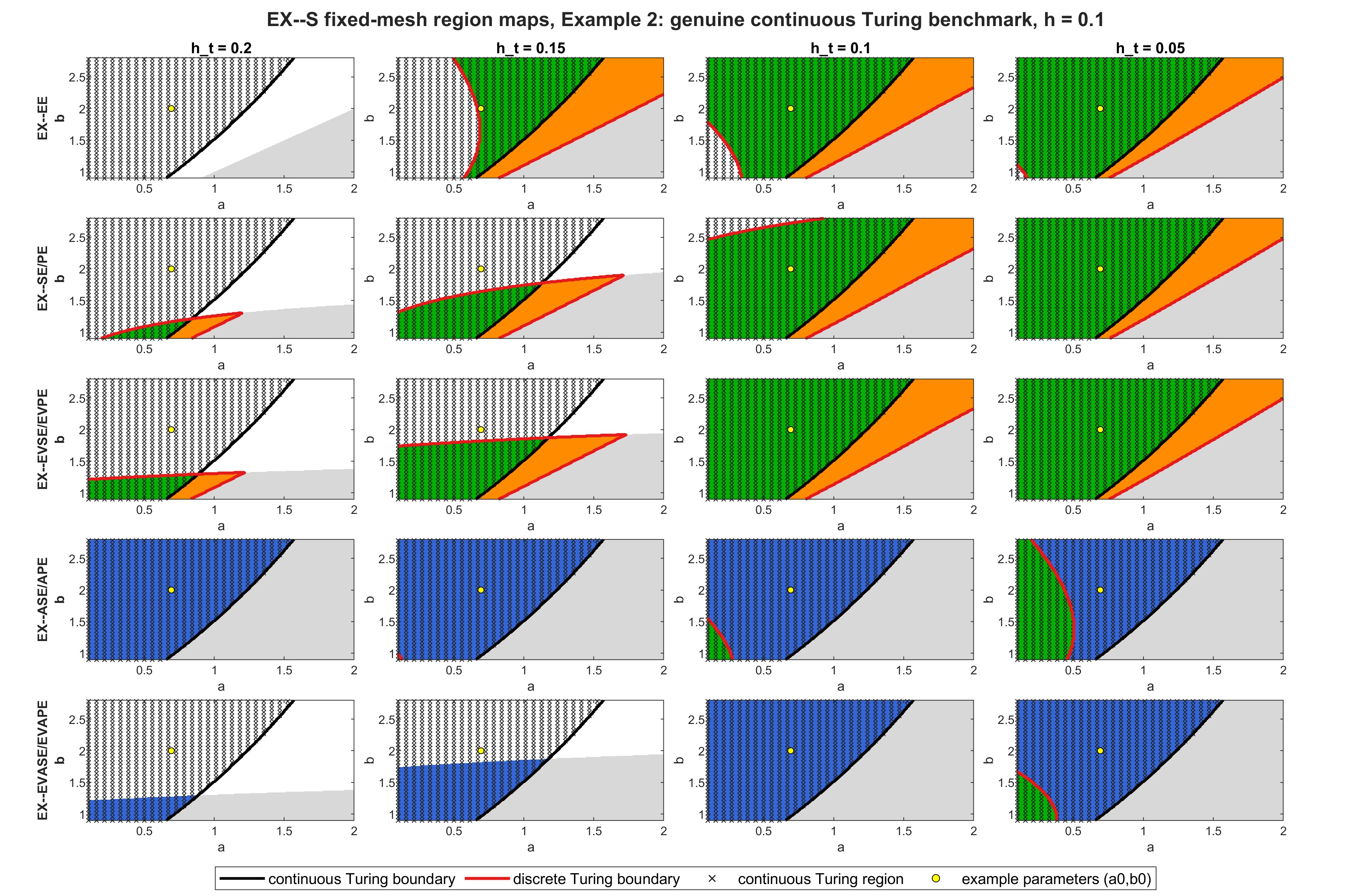}
\caption{EX--$\mathcal{S}$ fixed-mesh ($h=0.1$) algebraic region maps in the
neighbourhood of Example~2 for all five families, arranged as a grid: the rows
are EX--EE, EX--SE/PE, EX--EVSE/EVPE, EX--ASE/APE and EX--EVASE/EVAPE, while the
columns correspond to the decreasing time steps $h_t=0.2,\,0.15,\,0.1,\,0.05$.
The $(a,b)$-window is fixed across all panels, matching
Figure~\ref{fig:IE_maps_example2_all}, and colours follow the convention of
Table~\ref{tab:colour_convention}. The EX--EE, EX--SE/PE and
EX--EVSE/EVPE rows correctly detect the continuous Turing band (green, with the
red discrete boundary tracking the black continuous one), whereas the
EX--ASE/APE and EX--EVASE/EVAPE rows suppress it (blue). The blue regions
identify parameter values for which the continuous model is Turing unstable but
the EX--$\mathcal{S}$ scheme does not amplify any retained non-homogeneous mode; as $h_t$
decreases this delayed suppression shrinks.}
\label{fig:EX_maps_example2}
\end{figure}
\subsubsection{Summary for the EX--$\mathcal S$ family}
\label{subsubsec:EX_summary}
Because the EX--$\mathcal S$ Jury functions are generalized polynomials in
$x=e^{-h_t D_u\mu}$, see \eqref{eq:EX_generalized_polynomial_section7}, the
diagnostic is mesh-dependent and is read through the modal margins
$m^{\rm{EX}-\mathcal S}_{1}$ and $m^{\rm{EX}-\mathcal S}_{2}$ rather than through
parameter-only sign conditions. Two facts emerge from the modal diagnostics of
Figures~\ref{fig:EX_Fi_example1}--\ref{fig:EX_Fi_example2} and
from the parameter-plane maps of
Figures~\ref{fig:EX_maps_example1}--\ref{fig:EX_maps_example2}.

First, the qualitative dichotomy of the IE--$\mathcal S$ analysis survives the
exact treatment of diffusion. In the Turing-stable benchmark, EX--EE loses
non-homogeneous stability earliest, consistently with the unbounded transient
already seen in the nonlinear simulation, whereas EX--SE/PE retains the widest
admissible window. In the continuous Turing regime, EX--EE, EX--SE/PE and
EX--EVSE/EVPE drive $m^{\rm{EX}-\mathcal S}_{1,T}$ negative on the retained Turing
modes and detect the instability, while EX--ASE/APE and EX--EVASE/EVAPE keep it
positive and suppress the band.

Second, and this is the point specific to EX--$\mathcal S$, the second margin
$m^{\rm{EX}-\mathcal S}_{2}$ stays positive throughout the plotted range in
Example~2. Hence, the adjoint suppression is not an artefact of implicit
Euler diffusion: it is encoded in the reaction amplification matrix
$A_{h_t}^{\mathcal S}$ and persists when the diffusive subflow is exactly integrated. The
blue suppression regions of the EX--ASE/APE rows mirror their IE counterparts
and likewise shrink under refinement. The exact diffusion treatment changes the
modal damping factors, and therefore the precise extent of the unstable
regions, but not the family-level conclusion.

\subsection{Green-coverage sensitivity with respect to the time step}
\label{subsec:green_coverage_section7}

The parameter-plane maps discussed above show that the different schemes do
not preserve the continuous Turing region with the same accuracy when the time
step varies. We quantify this behaviour by focusing on a single indicator: the
\emph{green coverage}, namely the portion of the continuous Turing region, which is also detected as Turing unstable by the fully discrete scheme.

Let $\Omega_{ab}$ be the sampled parameter window and let
$\mathcal C\subset\Omega_{ab}$ denote the continuous Turing region. For a
scheme $\mathcal S$ and a time step $h_t$, let $\mathcal D^{\mathcal S}(h_t)$ be the corresponding
fully discrete Turing region. We define
\begin{equation}
\label{eq:green_coverage_section74}
G^\mathcal S(h_t)
=
\frac{|\mathcal C\cap \mathcal D^{\mathcal S}(h_t)|}{|\mathcal C|}.
\end{equation}
Thus $G^\mathcal{S}(h_t)\in[0,1]$. A value close to one means that most of the
continuous Turing region is retained by the fully discrete scheme, whereas a
value close to zero means that the continuous Turing region is almost entirely
missed. This index is therefore the quantitative counterpart of the green
portion of the parameter-plane maps.

In the following figures, $G^\mathcal S(h_t)$ is plotted for
$h_t\in[10^{-2},0.5]$. The horizontal axis is displayed in decreasing order:
larger time steps are placed on the left, while smaller time steps are placed
on the right. This representation makes it easier to read the progressive
recovery of the continuous Turing region as $h_t$ decreases.

\begin{figure}[tp]
\centering
\begin{minipage}[t]{0.48\textwidth}
\centering
\includegraphics[width=\textwidth]{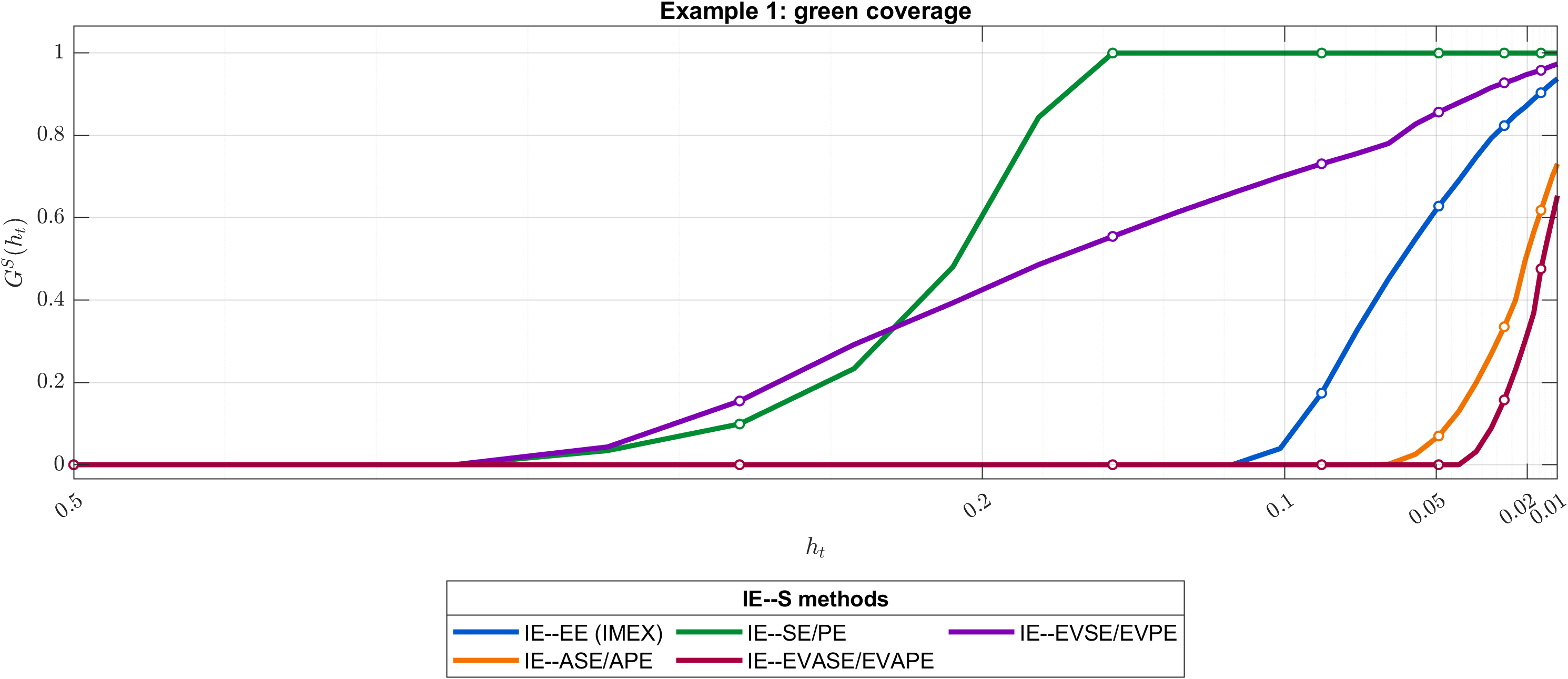}
\subcaption*{(a) Example 1}
\end{minipage}
\hfill
\begin{minipage}[t]{0.48\textwidth}
\centering
\includegraphics[width=\textwidth]{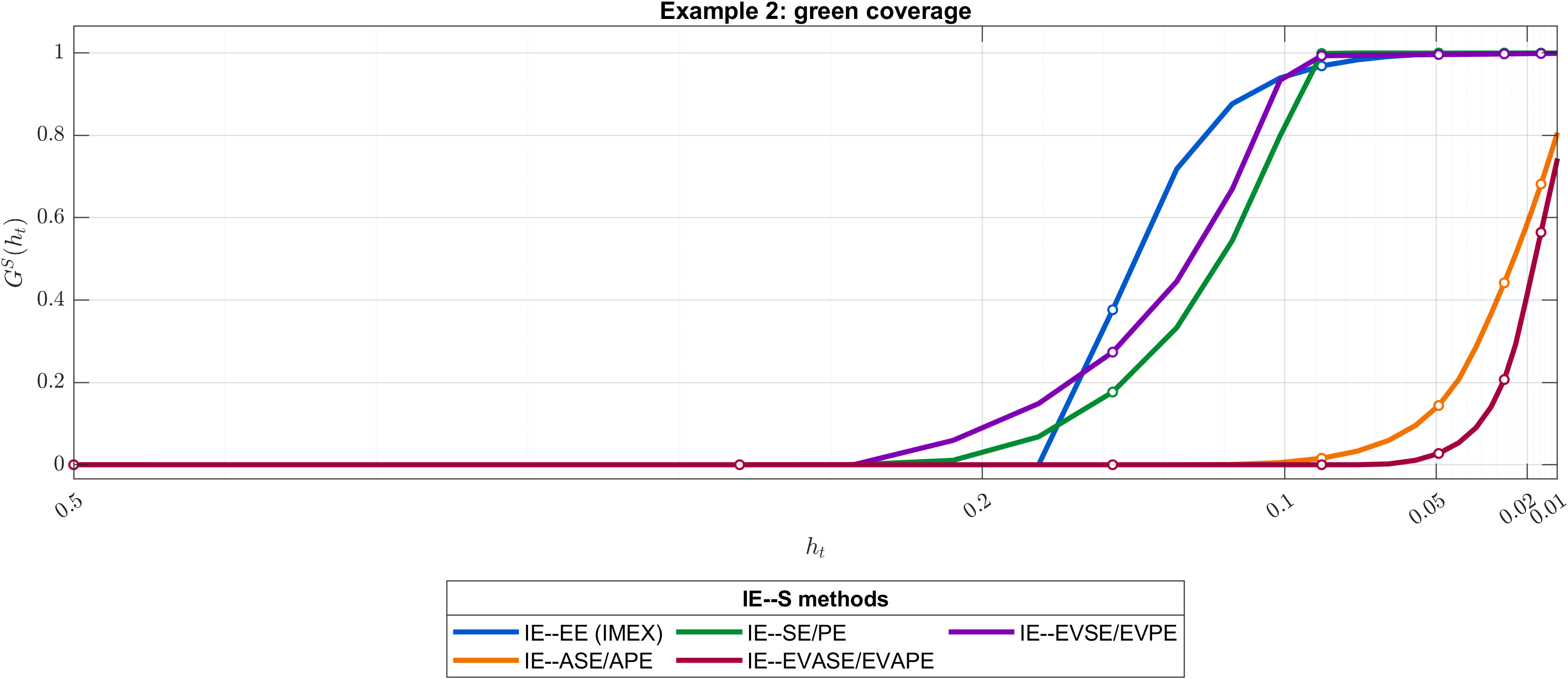}
\subcaption*{(b) Example 2}
\end{minipage}
\caption{Green coverage $G^\mathcal S(h_t)$ for the IE--$\mathcal{S}$ family. Panel (a) refers to
the neighbourhood of Example~1, while panel (b) refers to the 
continuous Turing regime of Example~2.}
\label{fig:green_coverage_IE}
\end{figure}

Figure~\ref{fig:green_coverage_IE} compares the IE--$\mathcal{S}$ schemes. In both
examples, the curves show how rapidly each method recovers the continuous
Turing region as the time step is reduced. In Example~1, the SE/PE and
EVSE/EVPE variants display a wider range of time steps for which the green
coverage is close to one, whereas the EE and adjoint variants lose coverage
more rapidly when $h_t$ increases. In Example~2, the same indicator highlights
the different behaviour of the adjoint-symplectic variants: ASE/APE and
EVASE/EVAPE require smaller time steps before recovering a substantial part of
the continuous Turing region, while SE/PE remains more robust over the tested
range.

\begin{figure}[htp]
\centering
\begin{minipage}[t]{0.48\textwidth}
\centering
\includegraphics[width=\textwidth]{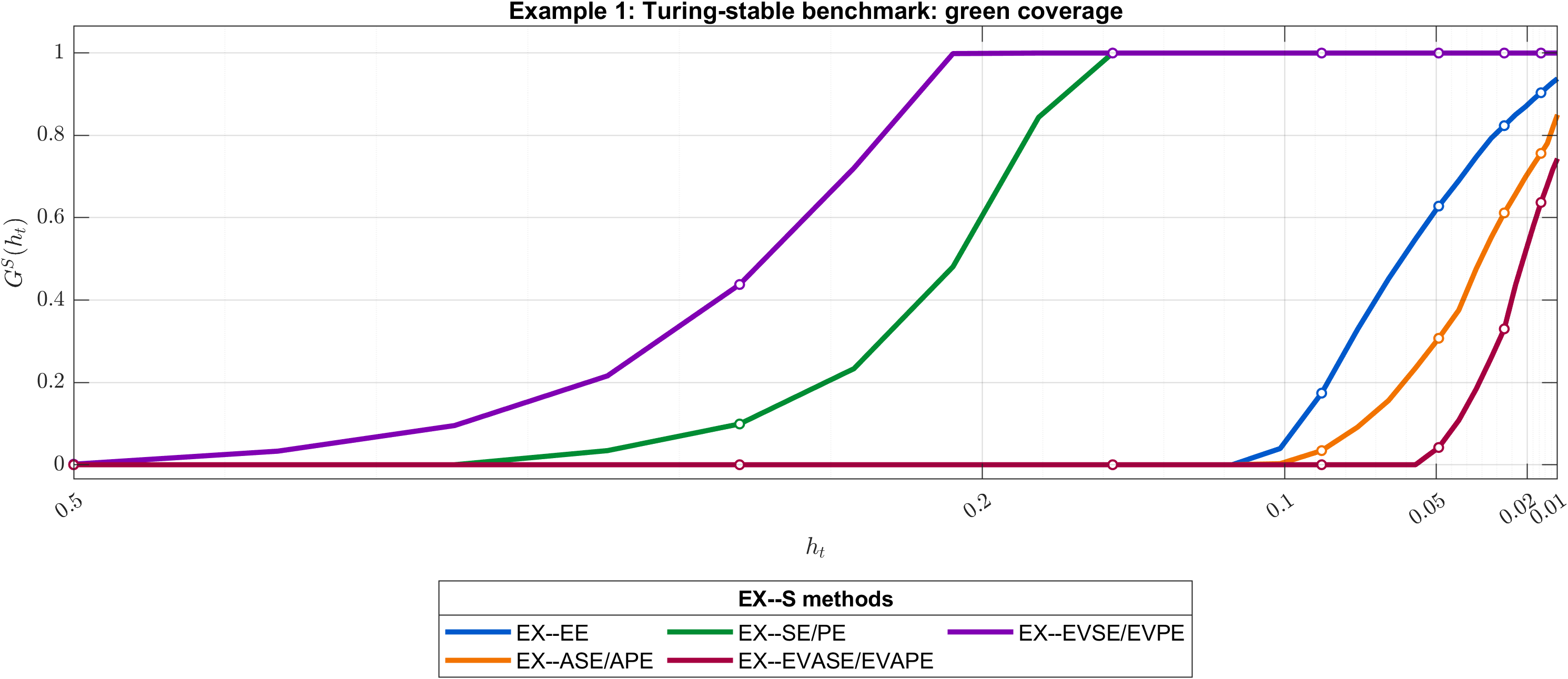}
\subcaption*{(a) Example 1}
\end{minipage}
\hfill
\begin{minipage}[t]{0.48\textwidth}
\centering
\includegraphics[width=\textwidth]{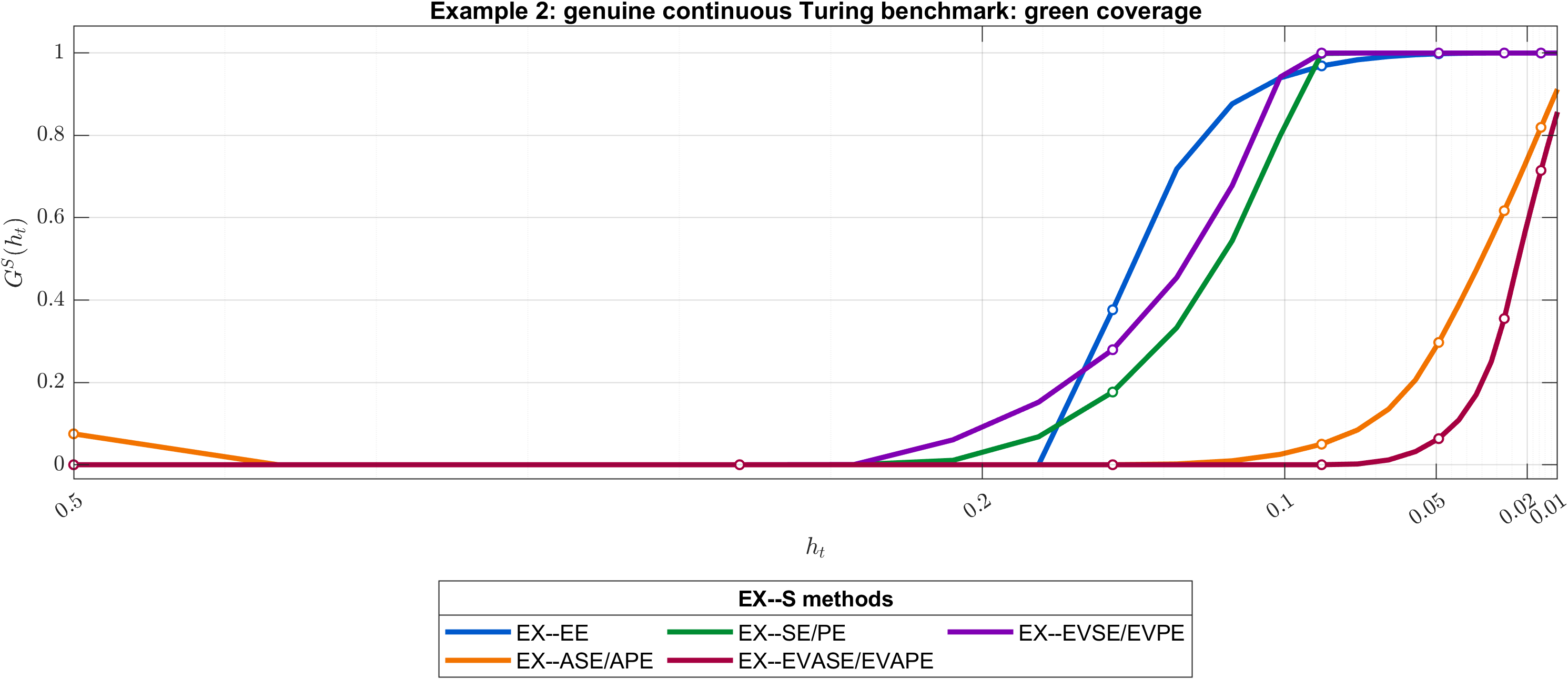}
\subcaption*{(b) Example 2}
\end{minipage}
\caption{Green coverage $G^\mathcal S(h_t)$ for the EX--$\mathcal{S}$ family. Panel (a) refers to
the neighbourhood of Example~1, while panel (b) refers to the 
continuous Turing regime of Example~2.}
\label{fig:green_coverage_EX}
\end{figure}

The EX--$\mathcal{S}$ results, reported in Figure~\ref{fig:green_coverage_EX}, confirm the
same qualitative picture. Since diffusion is treated exactly, these curves
mainly reflect the effect of the reaction update on the preservation of the
continuous Turing region. Again, the SE/PE and EVSE/EVPE variants retain a
larger green coverage for a wider range of time steps, whereas the adjoint
variants recover the continuous Turing region only when the time step becomes
sufficiently small. The behaviour of EX--EE is intermediate and depends on the
benchmark.

The green-coverage curves thus provide a compact quantitative counterpart of
the parameter-plane maps: a single scalar $G^\mathcal S(h_t)$ per scheme and
time step, increasing towards one as $h_t\to0$, whose ordering across families
is consistent in the IE--$\mathcal S$ and EX--$\mathcal S$ settings. The
comparative reading of this ordering is deferred to the synthesis of
Section~\ref{subsec:comparative_summary}.

\subsection{Comparative summary across diffusion treatments}
\label{subsec:comparative_summary}

Collecting the IE--$\mathcal S$ and EX--$\mathcal S$ analyses, three conclusions
are independent of how diffusive flow is treated.

The differences observed in the parameter-plane maps are governed primarily by
the reaction substep, not by the diffusion solver. IE--SE/PE and EX--SE/PE
behave alike, and so do the two adjoint families; switching from implicit Euler
to exact diffusion rescales the unstable regions but leaves the
green/orange/blue classification qualitatively unchanged. This is the
parameter-plane counterpart of the fact, established analytically in
Proposition~\ref{prop:continuous_interpretation}, that the leading
$\mathcal J_1$ contribution $h_t^2\,h(\mu)$ is scheme-independent.

The two pathologies are opposite and tied to different Jury
conditions. The EE-type schemes (IMEX) reproduce the continuous Turing sign
exactly through $\mathcal J_1$ but may create spurious non-homogeneous
instability through $\mathcal J_2$; the adjoint-symplectic schemes never violate
$\mathcal J_2$ in these regimes but fail to activate $\mathcal J_1$ on the
continuous Turing band. The green-coverage index $G^\mathcal S(h_t)$ of
Section~\ref{subsec:green_coverage_section7} condenses both effects into a
single curve and shows the same ordering in both settings.

The SE/PE family is the most balanced first-order reaction solver among those
considered. It avoids the $\mathcal J_2$-driven spurious pattern of IMEX in the
Turing-stable benchmark and the delayed suppression of the adjoint family in the
continuous Turing regime, retaining the largest green coverage over the broadest
range of admissible time steps, with both implicit and exact diffusion. However, we stress that this is an empirical conclusion on two benchmarks: as the
dominance analysis of Section~\ref{subsec:dominance} shows, no first-order scheme considered here provides an $h_t$-independent guarantee of
Turing-region preservation.

\section{Conclusions}
\label{sec:conclusions}
 
We have developed a fully discrete Turing instability analysis for
matrix-oriented first-order splitting methods applied to two-species
reaction--diffusion systems, and specialized the discussion to the
Gierer--Meinhardt activator--inhibitor model.  The main point of the paper is
that numerical pattern formation must be interpreted at the level of the fully
discrete map.  The continuous Turing polynomial \(h(\mu)\) identifies the
diffusion-driven instability of the differential problem, but the time
integrator introduces additional  conditions that may either create or
suppress non-homogeneous growth.  Each scheme therefore induces its own
discrete Turing region, and the relevant question is not whether the method is
stable, but whether this region coincides with the continuous one.  In the
language of geometric numerical integration, the diffusion-driven instability
is a qualitative property of the continuous flow, and a faithful integrator
should preserve it: it should be unstable on exactly the continuously unstable
modes and stable on the others.  The Jury-based analysis and the
parameter-plane maps developed here make this comparison explicit and
quantitative, so that the discrete Turing region of a candidate scheme can be
checked against the continuous one before the scheme is trusted to interpret
numerical patterns.
 
For IMEX, the first Jury condition preserves the sign of the continuous
Turing polynomial exactly, as shown in
Remark~\ref{prop:EE_J1_structure}.  This makes IMEX a reliable detector
of continuous Turing modes through \(\mathcal J_1\).  However, the
Gierer--Meinhardt examples show that this property does not exclude other
discrete artefacts: in a continuous Turing-stable regime, IMEX may lose
stability through \(\mathcal J_2\) and converge to a stable spurious numerical
pattern. Thus exact preservation of the Turing polynomial through
\(\mathcal J_1\) must be complemented by a check of the second Jury condition.
 
The adjoint-symplectic family displays the opposite pathology.  In the
continuous Turing regime considered here, IE--ASE/APE and EX--ASE/APE remain
stable on all modes even though the continuous model has a nonempty Turing
band. Their large homogeneous-mode stability thresholds are therefore
misleading: the methods are stable, but they are stable in a way that
suppresses the pattern-forming mechanism of the underlying PDE.  This confirms
that homogeneous stability is not an adequate proxy for pattern fidelity.
 
The SE/PE family plays an important role in the examples.  In the
continuous Turing-stable regime where IMEX produces a stable spurious
numerical pattern, IE--SE/PE and EX--SE/PE remain stable and damp the same
initial perturbation. This suggests that comparing families with different
reaction substeps can help distinguish a continuous instability from
a scheme-dependent artefact.
 
The dominance analysis sharpens this picture for IMEX. Because its first Jury
condition reproduces the sign of \(h(\mu)\) exactly, IMEX is Turing-region
preserving whenever the homogeneous mode is admissible in a genuine Turing
regime, and, in a Turing-stable regime, whenever an additional explicit
time-step bound controls the second Jury condition. The dominance of the second
Jury quantity is thus required only in the Turing-stable regime, where it rules
out the spurious pattern; it plays no role when the continuous model is
genuinely Turing unstable. What none of the first-order schemes provides is an
\(h_t\)-independent guarantee, valid uniformly in the time step. Thus, the
geometric requirement is not simply to enlarge stability bounds, but to design a
method whose discrete instability set is tied, by construction, to the
continuous Turing polynomial.
 
This points to a natural direction for future work: the construction of
integrators that preserve the continuous Turing region as a geometric
structure of the reaction--diffusion flow. Such a method should combine
homogeneous-mode stability with an exact or sign-preserving discretization of
the Turing polynomial, while ensuring that the remaining Jury conditions cannot
alter the continuous instability mechanism. In this sense, the present paper
does not identify a final optimal scheme; rather, it provides the algebraic
criteria that a Turing-region-preserving scheme should satisfy.
 
A further
direction, particularly relevant for the reaction kinetics considered here, is
the extension to positivity-preserving integrators. The Gierer--Meinhardt
reaction is intrinsically positive, and indeed in the EX--EE example the loss
of positivity of the numerical solution is precisely what drives the scheme to
break down. Geometric conservative nonstandard integrators of GeCo type
\cite{martiradonna2020geco} and modified Patankar--Runge--Kutta schemes
\cite{kopecz2018unconditionally} and their linear multistep extensions
\cite{izzo2025modified}, designed to preserve positivity (and, where
present, linear invariants) of biochemical and ecological systems
\cite{diele2020geometric, blanes2022positivity}, are natural candidates for a reaction substep within
the present splitting framework.  Carrying out the fully discrete Jury analysis
for such reaction maps would identify schemes that are not only
Turing-region-preserving in the sense developed above, but also unconditionally
positive, thereby ruling out the non-physical breakdown observed for the
explicit reaction step.  

Finally, the present Jury analysis is linear around
the homogeneous equilibrium, whereas the stable spurious patterns observed in
the IMEX example are nonlinear saturated states of the discrete map. A weakly nonlinear  analysis of the fully discrete dynamics
would clarify when a discrete-only instability saturates and when it instead
drives the computation away from the physically meaningful regime. Such an
analysis could build on the multiple-scales amplitude-equation approach
developed for continuous Turing patterns~\cite{bozzini2015weakly}, on recent
numerical bifurcation analyses of Turing and symmetry-broken patterns in
ecological reaction--diffusion models~\cite{spiliotis2026numerical}, and on
the observation that Turing instabilities alone do not necessarily ensure
sustained pattern formation~\cite{krause2024turing}. This perspective could also support the design of
interactive numerical environments for reaction--diffusion models, such as
VisualPDE~\cite{walker2023visualpde}, by complementing visual
exploration with scheme-dependent indicators of the corresponding discrete Turing regions.

\section*{Software availability}
\noindent
The MATLAB source code for the implementations used to compute the presented results can be downloaded from \url{https://github.com/CnrIacBaGit/Discrete_Turing}

\section*{Acknowledgements}

\noindent
F.D., C.M. and A.M. research activity is funded by PR PUGLIA FESR FSE+ 2021-2027 - Fondo Europeo Sviluppo Regionale - Asse Prioritario I ``Competitività e Innovazione'' - Obiettivo specifico RSO1.1 - Azione 1.5 ``Interventi per il rafforzamento del sistema innovativo regionale e sostegno alla collaborazione tra imprese e strutture di ricerca'' - Sub-Azione 1.5.1 ``Supporto alle attività di ricerca e sviluppo su aree tematiche di rilievo e all'applicazione di soluzioni tecnologiche funzionali alla realizzazione delle strategie di S3'', Avviso pubblico ``Reti - Sostegno alla ricerca collaborativa'' approvato con A.D. n. 208/2024, A.D. n. 216/2024, A.D. n. 227/2024, A.D. n. 230/2024 e AD n.3/2025, CUP B89J24003660007, project title ``PRISM- Diagnosi precoce e responsiva nei Vigneti Mediterranei''.

\noindent
F.D., C.M. and A.M. are members of the INdAM research group GNCS;  F.D., C.M. and A.M. would like to thank Mr. Cosimo Grippa for his valuable technical support.
\bibliographystyle{plain}
\bibliography{references.bib}

\end{document}